\documentclass[final]{siamltex}
\usepackage{epsfig}
\usepackage{color}
\usepackage{amsmath}
\usepackage{amscd}
\usepackage{amsmath, amsfonts, amssymb,mathrsfs}
\usepackage{multirow}
\usepackage{enumerate}
\usepackage{verbatim}
\usepackage{cite}
\usepackage{tikz}
\usetikzlibrary{matrix,arrows}
\usepackage[justification=raggedright]{caption}
\usepackage{subcaption}
\usepackage[ruled]{algorithm2e}

\graphicspath{{../FiguresUnboxedHelv/}}

\def \vec#1{{\bf{#1}}}
\def\pd#1#2{\frac{\partial #1}{\partial #2}}

\newcommand{\bi}{\begin{itemize}}
\newcommand{\ei}{\end{itemize}}

\newcommand{\diverg}{\vec{\nabla}\cdot}
\newcommand{\director}{\vec{n}}
\newcommand{\curl}{\vec{\nabla}\times}
\newcommand{\Ltwoinner}[3]{\langle #1,#2 \rangle_0}
\newcommand{\Ltwonorm}[2]{\Vert #1 \Vert_0}
\newcommand{\Ltwonormndim}[3]{\Vert #1 \Vert_0}
\newcommand{\Ltwoinnerndim}[4]{\langle #1,#2 \rangle_0}
\newcommand{\Rthree}{\mathbb{R}^3}

\newcommand{\diff}[1]{\, d#1}

\newcommand{\ltwonorm}[1]{\vert #1 \vert}
\newcommand{\ltwoinner}[2]{( #1, #2 )}
\newcommand{\kdirector}{\vec{n}_k}
\newcommand{\ddirector}{\delta \director}
\newcommand{\dlambda}{\delta \lambda}
\newcommand{\klambda}{\lambda_k}
\DeclareMathOperator*{\argmin}{argmin}

\newcommand{\lagdivn}{\mathcal{L}_{\director}[\vec{v}]}
\newcommand{\lagdivlam}{\mathcal{L}_{\lambda}[\gamma]}

\newcommand{\Honenorm}[2]{\Vert #1 \Vert_1}
\newcommand{\Hnnorm}[2]{\Vert #1 \Vert_#2}

\newcommand{\Hdc}{\mathcal{H}^{DC}{(\Omega)}}
\newcommand{\Hdcnot}{\mathcal{H}^{DC}_0{(\Omega)}}
\newcommand{\Hdcone}{\mathcal{H}^{DC^1}_0 {(\Omega)}}
\newcommand{\Hdconespace}{\mathcal{H}^{DC^1} {(\Omega)}}

\newcommand{\Hone}[1]{H^1(#1)}
\newcommand{\Hn}[2]{H^{#1}(#2)}
\newcommand{\Honenot}[1]{H^1_0({#1})}

\newcommand{\Ltwo}[1]{L^2(#1)}
\newcommand{\Ltwondim}[2]{L^2({#1})^{#2}}
\newcommand{\Lp}[1]{L^p (\Omega)}
\newcommand{\Hcurl}[1]{H(\text{curl},#1)}
\newcommand{\Hdiv}[1]{H(\text{div},#1)}
\newcommand{\Hdivnot}[1]{H_0(\text{div},#1)}
\newcommand{\Hcurlnot}[1]{H_0(\text{curl},#1)}
\newcommand{\Linfinity}[1]{L^{\infty}(\Omega)}
\newcommand{\Hdconenorm}[2]{\Vert #1 \Vert_{DC^1}}
\newcommand{\Hdcnorm}[2]{\Vert #1 \Vert_{DC}}

\newcommand{\triangulation}{\mathcal{T}_h}

\newcommand{\diam}{\text{diam }}

\newcounter{casenum}
\newenvironment{caseof}{\setcounter{casenum}{1}}{}
\newcommand{\case}[2]{\vskip.5\baselineskip\par\noindent {\bfseries Case \arabic{casenum}.} #1 \\ #2\addtocounter{casenum}{1}}

\newtheorem{assumption}[theorem]{Assumption}

\title{An Energy-Minimization Finite-Element Approach for the Frank-Oseen Model of Nematic Liquid Crystals: Continuum and Discrete Analysis}

\author{J. H. Adler, T. J. Atherton, D. B. Emerson, S. P. MacLachlan}

\begin{document}

\maketitle

\begin{abstract}
This paper outlines an energy-minimization finite-element approach to the computational modeling of equilibrium configurations for nematic liquid crystals under free elastic effects. The method targets minimization of the system free energy based on the Frank-Oseen free-energy model. Solutions to the intermediate discretized free elastic linearizations are shown to exist generally and are unique under certain assumptions. This requires proving continuity, coercivity, and weak coercivity for the accompanying appropriate bilinear forms within a mixed finite-element framework. Error analysis demonstrates that the method constitutes a convergent scheme. Numerical experiments are performed for problems with a range of physical parameters as well as simple and patterned boundary conditions. The resulting algorithm accurately handles heterogeneous constant coefficients and effectively resolves configurations resulting from complicated boundary conditions relevant in ongoing research.
 \end{abstract}



\pagestyle{myheadings}
\thispagestyle{plain}
\markboth{\sc Adler, Atherton, Emerson, MacLachlan}{\sc Nematic LC Energy Minimization}

\section{Introduction}

Liquid crystals, whose discovery is attributed to Reinitzer in 1888 \cite{Reinitzer1}, are substances that possess mesophases with properties intermediate between liquids and crystals. The mesophases exist at different temperatures or solvent concentrations. In recent years, research on the novel properties of liquid crystals has rapidly expanded. Modern applications include  nanoparticle organization, liquid crystal-functionalized polymer fibers \cite{Lagerwall1}, and liquid crystal elastomers designed to produce effective actuator devices such as light driven motors \cite{Yamada1} and artificial muscles \cite{Thomsen1}.

The focus of this paper is on nematic liquid crystal phases, which are formed by rod-like molecules that self-assemble into an ordered structure, such that the molecules tend to align along a preferred orientation. The preferred average direction at any point in a domain, $\Omega$, is known as the director, denoted $\director(x,y,z) = (n_1, n_2, n_3)^T$. The director is taken to be of unit length at every point and headless, that is $\director$ and $-\director$ are indistinguishable, reflecting the observed experimental symmetry of the phase. 

In addition to their self-structuring properties, the orientation of a nematic liquid crystal may be affected by applied electric fields. Moreover, since these materials are birefringent, that is these materials' refractive indices depend on the polarization of light, they can be used to control the propagation of light through a nematic structure. These traits have led, and continue to lead, to important discoveries in display technologies and beyond \cite{Lagerwall1}. Thorough overviews of liquid crystal physics and properties are found in \cite{Stewart1, deGennes1, Chandrasekhar1}. 

Many mathematical and computational models of liquid crystal continuum theory lead to complicated systems involving unit length constrained vector fields. Currently, the complexity of such systems has restricted the existence of known analytical solutions to simplified geometries in one (1-D) or two-dimensions (2-D), often under strong simplifying assumptions. When coupled with electric fields and other effects, far fewer analytical solutions exist, even in 1-D \cite{Stewart1}. In addition, associated systems of partial differential equations, such as the equilibrium equations \cite{Stewart1, Ericksen4}, suffer from non-unique solutions, which must be distinguished via energy arguments. Due to such difficulties, efficient, theoretically supported, numerical approaches to the modeling of nematic liquid crystals under free elastic and augmented electric effects are of great importance. This paper discusses the modeling of free elastic effects. The addition of electric field effects will be the subject of future work. A number of computational techniques for liquid crystal equilibrium and dynamics problems exist \cite{Liu1, Liu2, Liu3, Stewart1}, including least-squares finite-element methods \cite{Atherton1} and discrete Lagrange multiplier approaches \cite{Ramage1, Ramage2}. 

In this paper, we propose a method that directly targets energy minimization in the continuum, via Lagrange multiplier theory on Banach spaces. The approach is derived absent the often used one-constant approximations \cite{Ramage1, Liu1, Liu2, Liu3, Ramage2, Stewart1, Cohen1}; that is, the method described here and the accompanying theory are applied for a wide range of physical parameters. This allows for significantly improved modeling of physical phenomena not captured in many models. Furthermore, most models and analytical approaches rely on assumptions to reduce the dimensionality of the problem. Here, the method and theory are suitable for use on 2-D and 3-D domains and are easily combined with additional energy effects. 

After defining the energy functional to be minimized, first-order optimality conditions are computed. These first-order conditions contain highly nonlinear terms and are, therefore, linearized with a generalized Newton's method. The resulting Newton linearization resembles a typical mixed finite-element method formulation \cite{Brenner1, Braess1, Boffi1}. However, these forms present unique difficulties not found, for instance, in the Stokes' problem. In particular, the forms related to the nonlinear unit-length constraint for $\director$ require novel treatment.  Additionally, the proofs of continuity and coercivity differ significantly from many standard approaches due to the inherent complexity of the bilinear forms. 

In the continuum, it is possible to demonstrate coercivity for the relevant bilinear form with moderate simplifying assumptions. With auxiliary regularity assumptions, continuity of the involved bilinear forms is also established. On the other hand, for a pair of discrete spaces, continuity, coercivity, and weak coercivity for the relevant bilinear forms are proved. The main result of this paper proves the existence and uniqueness of solutions to each discrete Newton iteration. Error analysis is also performed to elaborate the convergence order of the approximations. The method is implemented and run for a number of configurations, including those relevant to ongoing research.

This paper is organized as follows. We first introduce the liquid crystal model under consideration, derive the method, and discuss Dirichlet boundary condition simplifications in Section \ref{energymodels}. In Section \ref{existuniquesection}, well-posedness of the Newton iterations for a pair of discrete spaces is proved and an error analysis is performed. The numerical methodology and numerical experiments are detailed in Section \ref{nummethodology}. Finally, Section \ref{conclusion} gives some concluding remarks and future work is discussed.

\section{Energy Model} \label{energymodels}

At equilibrium, absent any external forces, fields, or boundary conditions, the free elastic energy present in a liquid crystal sample is given by an integral functional, $\mathcal{F}$, which depends on the state variables of the system. A liquid crystal sample tends to the state of lowest free energy. While a number of free-energy models exist cf. \cite{Davis1}, this paper considers the Frank-Oseen free elastic model \cite{Stewart1, Virga1, Frank1}. The Frank-Oseen equations represent the free elastic energy density, $w_F$, in a sample as
\begin{equation*} \label{FrankOseenFree}
w_F = \frac{1}{2}K_1(\diverg \director)^2+ \frac{1}{2}K_2(\director \cdot \curl \director)^2+ \frac{1}{2}K_3\vert \director \times \curl \director \vert^2 + \frac{1}{2}(K_2+K_4)\diverg[(\director \cdot \vec{\nabla}) \director - (\diverg \director) \director].
\end{equation*}
Throughout this paper, the standard Euclidean inner product and norm are denoted $(\cdot, \cdot)$ and $\vert \cdot \vert$, respectively. The $K_i$, $i=1,2,3,4$, are known as the Frank elastic constants \cite{Frank1}, which vary depending on temperature and liquid crystal type. By Ericksen's inequalities \cite{Ericksen2}, $K_j \geq 0$ for $j=1,2,3$. Each term represents an energy penalty for the presence of splay, twist, bend, and saddle-splay, respectively. 

It can be shown that
\begin{equation} \label{SaddleSplayIdentity}
\diverg[(\director \cdot \vec{\nabla}) \director - (\diverg \director) \director] = \nabla n_1 \cdot \frac{\partial \director}{\partial x}+\nabla n_2 \cdot \frac{\partial \director}{\partial y}+\nabla n_3 \cdot \frac{\partial \director}{\partial z}-(\diverg \director)^2.
\end{equation}
Additionally, let
\begin{equation} \label{matrixD}
\vec{Z} = \kappa \director \otimes \director + (\vec{I} - \director \otimes \director) = \vec{I} - (1-\kappa) \director \otimes \director,
\end{equation}
where $\kappa = K_2/K_3$; in general, we consider the case that  $K_2, K_3 > 0$. Denote the classical $\Ltwo{\Omega}$ inner product and norm as $\Ltwoinner{\cdot}{\cdot}{\Omega}$ and $\Ltwonorm{\cdot}{\Omega}$, respectively. Employing \eqref{SaddleSplayIdentity}, \eqref{matrixD}, and the fact that $\director$ has unit length, the total free energy for a domain, $\Omega$, is
\begin{align}
\int_{\Omega} w_F \diff{V} &=\frac{1}{2}(K_1-K_2-K_4) \Ltwonorm{\diverg \director}{\Omega}^2 + \frac{1}{2}K_3\Ltwoinnerndim{\vec{Z} \curl \director}{\curl \director}{\Omega}{3} \nonumber \\
& \qquad + \frac{1}{2}(K_2+K_4) \big(\Ltwoinnerndim{\nabla n_1}{\frac{\partial \director}{\partial x}}{\Omega}{3}+ \Ltwoinnerndim{\nabla n_2}{\frac{\partial \director}{\partial y}}{\Omega}{3}+ \Ltwoinnerndim{\nabla n_3}{\frac{\partial \director}{\partial z}}{\Omega}{3} \big). \label{SystemEnergy}
\end{align} 

For the special case of full Dirichlet boundary conditions, we consider a fixed director $\director$ at each point on the boundary of $\Omega$. Considering the integration carried out on the terms in \eqref{SaddleSplayIdentity},
\begin{align} \label{stronganchoringdivthm}
&\frac{1}{2}(K_2+K_4) \int_{\Omega} \diverg[(\director \cdot \vec{\nabla}) \director - (\diverg \director) \director] \diff{V} \nonumber \\
&\qquad \qquad \qquad = \frac{1}{2}(K_2+K_4)\int_{\partial \Omega} [(\director \cdot \vec{\nabla}) \director - (\diverg \director) \director] \cdot \mathbf{\nu} \diff{S},
\end{align}
by the divergence theorem. Further, since $\director$ is fixed along $\partial \Omega$, the energy contributed by $\director$ on the boundary is constant regardless of the configuration of $\director$ on the interior of $\Omega$. Thus, in the minimization to follow, the energy contribution from this term is ignored. For this reason, \eqref{stronganchoringdivthm} is often referred to as a null Lagrangian \cite{Virga1}.

A number of methods involving computation of liquid crystal equilibria or dynamics utilize the so called one-constant approximation that $K_1=K_2=K_3$ and $K_4 = 0$ \cite{Ramage1, Liu1, Stewart1, Cohen1}, in order to significantly simplify the free elastic energy density to
\begin{equation*}
\hat{w}_F = \frac{1}{2}K_1 \vert \nabla \director \vert^2, \qquad \text{where } \vert \nabla \director \vert^2 = \sum_{i,j=1}^3 \left ( \pd{n_i}{x_j} \right)^2.
\end{equation*}
This expression for the free energy density is more amenable to theoretical development but ignores significant physical characteristics of the nematic \cite{Lee1, Atherton2}. The following method is derived without such an assumption.

\subsection{Free Elastic Energy Minimization} \label{freeenergymin}

In this section, a general approach for computing the free elastic equilibrium state for $\director$ is derived. This equilibrium state corresponds to the configuration which minimizes the system free energy subject to the local constraint that $\director$ is of unit length throughout the sample volume, $\Omega$. That is, the minimizer must satisfy $\director \cdot \director = 1$ pointwise throughout the volume. To compute this state, define the functional, equivalent to \eqref{SystemEnergy},
\begin{align} \label{functional2}
\mathcal{F}_1(\director) &= (K_1-K_2-K_4) \Ltwonorm{\diverg \director}{\Omega}^2 + K_3\Ltwoinnerndim{\vec{Z} \curl \director}{\curl \director}{\Omega}{3} \nonumber \\
& \qquad + (K_2+K_4) \big(\Ltwoinnerndim{\nabla n_1}{\frac{\partial \director}{\partial x}}{\Omega}{3}+ \Ltwoinnerndim{\nabla n_2}{\frac{\partial \director}{\partial y}}{\Omega}{3}+ \Ltwoinnerndim{\nabla n_3}{\frac{\partial \director}{\partial z}}{\Omega}{3} \big).
\end{align}
Define 
\begin{align*}
\Hdiv{\Omega} &= \{\vec{v} \in L^2(\Omega)^3 : \diverg \vec{v} \in L^2(\Omega) \},\\
\Hcurl{\Omega} &= \{ \vec{v} \in L^2(\Omega)^3 : \curl \vec{v} \in L^2(\Omega)^3 \}.
\end{align*}
Further, let
\begin{align*}
\Hdivnot{\Omega} &= \{\vec{v} \in \Hdiv{\Omega} : \mathbf{\nu} \cdot \vec{v} = 0 \text{ on } \partial \Omega\}, \\
\Hcurlnot{\Omega} &= \{\vec{v} \in \Hcurl{\Omega} : \mathbf{\nu} \times \vec{v} = \vec{0} \text{ on } \partial \Omega\},
\end{align*}
where $\mathbf{\nu}$ is the outward unit normal for $\partial \Omega$. Define 
\begin{equation*}
\Hdc= \{ \vec{v} \in \Hdiv{\Omega} \cap \Hcurl{\Omega} : B(\vec{v}) = g \},
\end{equation*}
with norm $\Hdcnorm{\vec{v}}{\Omega}^2 = \Ltwonormndim{\vec{v}}{\Omega}{3}^2 + \Ltwonorm{\diverg \vec{v}}{\Omega}^2 + \Ltwonormndim{\curl \vec{v}}{\Omega}{3}^2$ and appropriate boundary conditions $B(\vec{v})=g$. Further, let $\Hdcnot = \{ \vec{v} \in \Hdiv{\Omega} \cap \Hcurl{\Omega} : B(\vec{v}) = \vec{0} \}$. Finally,  denote the unit sphere as $\mathcal{S}^2$. The desired minimization becomes
\begin{equation*} \label{minioversphere}
\director_{*} = \argmin_{\director \in \mathcal{S}^2 \cap \Hdc} \mathcal{F}_1(\director).
\end{equation*}

In the presence of full Dirichlet boundary conditions, the functional to be minimized is significantly simplified as
\begin{equation} \label{functional3}
\mathcal{F}_2(\director) = K_1 \Ltwonorm{\diverg \director}{\Omega}^2 + K_3\Ltwoinnerndim{\vec{Z} \curl \director}{\curl \director}{\Omega}{3},
\end{equation}
by the application of \eqref{stronganchoringdivthm}. However, the functional still contains nonlinear terms introduced by the presence of $\vec{Z} = \vec{Z}(\director)$. Note that this simplification is also applicable to a rectangular domain with mixed Dirichlet and periodic boundary conditions. Such a domain is considered in the numerical experiments presented here. 

We proceed with the functional in \eqref{functional2} in building a framework for minimization under general boundary conditions. However, in the treatment of existence and uniqueness theory, we assume the application of full Dirichlet or mixed Dirichlet and periodic boundary conditions and, therefore, utilize the simplified form in \eqref{functional3}.

\subsection{First-Order Optimality Conditions and Newton Linearization} \label{newtonstepssection}

Since $\director$ must be of unit length, it is natural to employ a Lagrange multiplier approach. This length requirement represents a pointwise equality constraint, such that $\ltwoinner{\director}{\director} - 1 = 0$. Thus, following general constrained optimization theory \cite{Luenberger1}, define the Lagrangian
\begin{align*}
\mathcal{L}(\director, \lambda) &= \mathcal{F}_1(\director) + \int_{\Omega} \lambda(\vec{x})(\ltwoinner{\director}{\director}-1) \diff{V},
\end{align*}
where $\lambda \in \Ltwo{\Omega}$. In order to minimize \eqref{functional2}, we compute the G\^{a}teaux derivatives of $\mathcal{L}$ with respect to $\director$ and $\lambda$ in the directions $\vec{v} \in \Hdcnot$ and $\gamma \in L^2(\Omega)$, respectively. Hence, the necessary continuum first-order optimality conditions are
\begin{align}
\lagdivn &= \frac{\partial}{\partial \director} \mathcal{L}(\director, \lambda) [\vec{v}] =0, & & \forall \vec{v} \in \Hdcnot, \label{lagrangeweakforminitial} \\
\lagdivlam &= \frac{\partial}{\partial \lambda} \mathcal{L}(\director, \lambda) [\gamma] =0,& & \forall \gamma \in L^2(\Omega). \label{constraintinitial}
\end{align}
Computing these derivatives yields
\begin{align*}
\lagdivn = &2(K_1-K_2-K_4)\Ltwoinner{\diverg \director}{\diverg \vec{v}}{\Omega} + 2K_3\Ltwoinnerndim{\vec{Z}(\director) \curl \director}{\curl \vec{v}}{\Omega}{3} \nonumber \\
&+ 2(K_2-K_3)\Ltwoinner{\director \cdot \curl \director}{\vec{v} \cdot \curl \director}{\Omega} + 2(K_2+K_4)\big(\Ltwoinnerndim{\nabla n_1}{\pd{\vec{v}}{x}}{\Omega}{3} \nonumber \\ 
&+\Ltwoinnerndim{\nabla n_2}{\pd{\vec{v}}{y}}{\Omega}{3} +\Ltwoinnerndim{\nabla n_3}{\pd{\vec{v}}{z}}{\Omega}{3} \big) +2 \int_{\Omega} \lambda \ltwoinner{\director}{\vec{v}} \diff{V},
\end{align*}
and
\begin{align*}
\lagdivlam&= \int_{\Omega} \gamma (\ltwoinner{\director}{\director} -1) \diff{V}.
\end{align*}

The variational system contains nonlinearities in both \eqref{lagrangeweakforminitial} and \eqref{constraintinitial}. Therefore, Newton iterations are employed by computing a generalized first-order Taylor series expansion, requiring computation of the Hessian \cite{Benzi1, Nocedal1}.

Let $\kdirector$ and $\lambda_k$ be the current approximations for $\director$ and $\lambda$, respectively. Additionally, let $\ddirector= \director_{k+1} - \kdirector$ and $\dlambda = \lambda_{k+1}-\klambda$ be updates to these approximations. Then, the Newton iterations are denoted
 \begin{equation} \label{newtonhessian}
\left [ \begin{array}{c c}
\mathcal{L}_{\director \director} &  \mathcal{L}_{\director \lambda} \\
\mathcal{L}_{\lambda \director} & \mathcal{L}_{\lambda \lambda}
\end{array} \right ]
\left [ \begin{array}{c} 
\ddirector \\
\dlambda
\end{array} \right]
= - \left[ \begin{array}{c}
\mathcal{L}_{\director} \\
\mathcal{L}_{\lambda} 
\end{array} \right],
\end{equation}
where each of the system components are evaluated at $\kdirector$ and $\klambda$. The matrix-vector multiplication indicates the direction that the derivatives in the Hessian are taken. That is,
\begin{align*}
\mathcal{L}_{\director \director}[\vec{v}] \cdot \ddirector = \pd{ }{\director} \left ( \mathcal{L}_{\director} (\kdirector, \klambda)[\vec{v}]\right)[\ddirector], & &
\mathcal{L}_{\director \lambda}[\vec{v}] \cdot \dlambda = \pd{ }{\lambda} \left( \mathcal{L}_{\director} (\kdirector, \klambda)[\vec{v}] \right ) [\dlambda], \\
\mathcal{L}_{\lambda \director}[\gamma] \cdot \ddirector = \pd{ }{\director} \left( \mathcal{L}_{\lambda} (\kdirector, \klambda)[\gamma] \right)[\ddirector], & &
\mathcal{L}_{\lambda \lambda}[\gamma] \cdot \dlambda = \pd{ }{\lambda} \left( \mathcal{L}_{\lambda} (\kdirector, \klambda)[\gamma] \right)[\dlambda],
\end{align*}
where the partials denote G\^{a}teaux derivatives in the respective variables. 

Since $\mathcal{L}(\director, \lambda)$ is linear in $\lambda$, $\mathcal{L}_{\lambda \lambda}[\gamma] \cdot \dlambda = 0$. Hence, the Hessian in $\eqref{newtonhessian}$ simplifies to a saddle-point structure,
\begin{equation*}
\left [ \begin{array}{c c}
\mathcal{L}_{\director \director} &  \mathcal{L}_{\director \lambda} \\
\mathcal{L}_{\lambda \director} & \vec{0}
\end{array} \right ].
\end{equation*}
The discrete form of this Hessian leads to a saddle-point matrix, which poses unique difficulties for the efficient computation of the solution to the resulting linear system. Such structures commonly appear in constrained optimization and other settings; for a comprehensive overview of discrete saddle-point problems see \cite{Benzi2}. Here, we focus only on the linearization step rather than the underlying linear solvers, which will be investigated in future work. Computing the remaining G\^{a}teaux derivatives yields
\begin{align}
\mathcal{L}_{\director \lambda}[\vec{v}] \cdot \dlambda &=  2 \int_{\Omega} \dlambda \ltwoinner{\kdirector}{\vec{v}} \diff{V} \label{freeElasticLagnlam},\\
\mathcal{L}_{\lambda \director}[\gamma] \cdot \ddirector &= 2 \int_{\Omega} \gamma \ltwoinner{\kdirector}{\ddirector} \diff{V} \label{freeElasticLaglamn},
\end{align}
and
\begin{align}
\mathcal{L}_{\director \director}[\vec{v}] \cdot \ddirector =&2(K_1 - K_2 - K_4)\Ltwoinner{\diverg \ddirector}{\diverg \vec{v}}{\Omega} + 2K_3 \Ltwoinnerndim{\vec{Z}(\kdirector) \curl \ddirector}{\curl \vec{v}}{\Omega}{3} \nonumber \\
& + 2(K_2 - K_3) \Big(\Ltwoinner{\ddirector \cdot \curl \vec{v}}{\kdirector \cdot \curl \kdirector}{\Omega}+\Ltwoinner{\kdirector \cdot \curl \vec{v}}{\ddirector \cdot \curl \kdirector}{\Omega} \nonumber \\
& + \Ltwoinner{\kdirector \cdot \curl \kdirector}{\vec{v} \cdot \curl \ddirector}{\Omega} + \Ltwoinner{\kdirector \cdot \curl \ddirector}{\vec{v} \cdot \curl \kdirector}{\Omega} \nonumber \\
& + \Ltwoinner{\ddirector \cdot \curl \kdirector}{\vec{v} \cdot \curl \kdirector}{\Omega}\Big) + 2(K_2+K_4) \big( \Ltwoinnerndim{\nabla \delta n_1}{\pd{\vec{v}}{x}}{\Omega}{3} \nonumber \\
& +\Ltwoinnerndim{\nabla \delta n_2}{\pd{\vec{v}}{y}}{\Omega}{3} + \Ltwoinnerndim{\nabla \delta n_3}{\pd{\vec{v}}{z}}{\Omega}{3} \big) + 2\int_{\Omega} \klambda \ltwoinner{\ddirector}{\vec{v}} \diff{V}\label{freeElasticLagnn}.
\end{align}
Constructing \eqref{newtonhessian} using \eqref{freeElasticLagnlam}-\eqref{freeElasticLagnn} yields a linearized variational system. For these iterations, we compute $\ddirector$ and $\dlambda$ satisfying this system for all $\vec{v} \in \Hdcnot$ and $\gamma \in L^2(\Omega)$ with the current approximations $\kdirector$ and $\klambda$. The current approximations are then corrected with the solutions $\ddirector$ and $\dlambda$ to yield $\director_{k+1}$ and $\lambda_{k+1}$. While they typically improve robustness and efficiency, we do not consider the use of line searches or trust regions \cite{Nocedal1} in the work presented here, leaving this for future work.

If we are considering a system with Dirichlet or mixed periodic and Dirichlet boundary conditions, as described above, we eliminate the $(K_2 + K_4)$ terms from \eqref{newtonhessian}, simplifying the linearization.

\subsection{Uniform Symmetric Positive Definiteness of \vec{Z}}

In subsequent sections, theory establishing the existence and uniqueness of solutions to the Newton linearizations is developed. A key property exploited in these proofs is that $\vec{Z}$ is uniformly symmetric positive definite (USPD) under reasonable assumptions.

It is relatively routine to show that $\vec{Z}$ is symmetric, self-adjoint in $\Ltwondim{\Omega}{3}$, and has, at each point in $\Omega$, eigenvalues $\mu = 1,1, 1+(\kappa-1)(n_1^2 + n_2^2+n_3^2)$. Ericksen's inequalities \cite{Ericksen2} guarantee that $K_2, K_3 \geq 0$. Throughout this paper, we consider the case where the inequality is strict; thus, $\kappa > 0$. 
We also assume that, in the Newton iterations, control has been maintained over the director length such that
\begin{equation} \label{limitsonnlength}
\alpha \leq n_1^2 + n_2^2+n_3^2 \leq \beta, \qquad \forall \vec{x} \in \Omega,
\end{equation}
with constants $0 < \alpha \leq 1 \leq \beta$.
\begin{lemma} \label{USPDlemma}
Assume that $\alpha \leq \ltwoinner{\director}{\director} \leq \beta \text{ for all } \vec{x} \in \Omega$. If $\kappa>1$, then $\vec{Z}$ is USPD on $\Omega$. For $0<\kappa<1$, if $\beta<\frac{1}{1-\kappa}$, then $\vec{Z}$ is USPD on $\Omega$.
\end{lemma}
\begin{proof}
For a fixed $\vec{x} \in \Omega$, note that 
\begin{align*}
\ltwonorm{\vec{Z}(\vec{x})} &=\max_{1 \leq i \leq 3} \mu_i(\vec{Z}(\vec{x}) \label{spectralnormeq}),
\end{align*}
where $\mu_i(\vec{Z}(\vec{x}))$ denotes the $i^{\text{th}}$ eigenvalue of $\vec{Z}$. In order to keep the eigenvalues of $\vec{Z}$ positive, it is neccessary that $\mu_3 = (1+(\kappa-1)(n_1^2 + n_2^2+n_3^2)) > 0$. We consider two cases.
\begin{caseof}
\case{$\kappa >1$.}{
If $\kappa >1$, then \begin{equation}
0<1+(\kappa-1) \alpha \leq \mu_3 \leq 1+(\kappa-1)\beta, \qquad \forall \vec{x} \in \Omega. \label{kappalargeeigenineq}
\end{equation}
Thus, \eqref{kappalargeeigenineq} implies that the eigenvalues of $\vec{Z}(\vec{x})$ are bounded by 
\begin{equation}
1 \leq \mu_i \leq 1+(\kappa-1)\beta \label{bigkapbound}, \qquad \forall \vec{x} \in \Omega.
\end{equation}
Using standard functional analysis arguments \cite{Griffel1}, Inequality \eqref{bigkapbound} implies that
\begin{equation} \label{bigkappaUSPD}
1 \leq \frac{\vec{\xi}^T \vec{Z}(\vec{x}) \vec{\xi}}{\vec{\xi}^T \vec{\xi}} \leq 1+(\kappa-1)\beta, \qquad \forall \vec{x} \in \Omega, \vec{\xi} \in \Rthree.
\end{equation}
}
\case{$0< \kappa <1$.}{
For this case,
\begin{equation*}
1+(\kappa-1)\beta \leq \mu_3 \leq 1+(\kappa-1) \alpha.
\end{equation*}
However, the assumption that $\beta < \frac{1}{1-\kappa}$ implies that
\begin{equation*}
0 < 1+(\kappa-1)\beta \leq \mu_3 \leq 1+(\kappa-1) \alpha \label{kappasmalleigenineq}.
\end{equation*}
Hence, the eigenvalues of $\vec{Z}(\vec{x})$ are bounded by
\begin{equation}
0 < 1+(\kappa -1)\beta \leq \mu_i \leq 1 \label{smallkapbound}, \qquad \forall \vec{x} \in \Omega.
\end{equation}
As in the previous case, \eqref{smallkapbound} implies
\begin{equation}
1+(\kappa-1)\beta \leq \frac{\vec{\xi}^T \vec{Z}(\vec{x}) \vec{\xi}}{\vec{\xi}^T \vec{\xi}} \leq 1, \qquad \forall \vec{x} \in \Omega, \vec{\xi} \in \Rthree.\label{smallkappaUSPD}
\end{equation}
}
\end{caseof}
\noindent Thus, $\vec{Z}$ is USPD for any $\kappa>0$, as long as sufficient control is maintained on the length of $\director$.
\end{proof}

The USPD property of $\vec{Z}$ plays an important role in the proofs of existence and uniqueness of solutions to the linearization undertaken in the next section.

\section{Existence and Uniqueness for the Newton Linearizations} \label{existuniquesection} 

Here and in the following subsections, we will routinely make use of the following set of assumptions.
\begin{assumption} \label{secass}
Consider an open bounded domain, $\Omega$, which is a convex polyhedron or has a $C^{1,1}$ boundary. Note that this implies that the boundary is also Lipschitz continuous. Further, assume that $\alpha \leq \ltwonorm{\kdirector}^2 \leq \beta$, such that $\vec{Z}(\kdirector(\vec{x}))$ remains USPD with lower and upper bounds, $\eta$ and $\Lambda$, respectively. Finally, Dirichlet boundary conditions are applied. Therefore, both $\ddirector$ and $\vec{v}$ are in $\Hdivnot{\Omega} \cap \Hcurlnot{\Omega}$.
\end{assumption}

In the continuum, the above Newton systems are written in a general form as
\begin{align}
a(\ddirector, \vec{v}) + b(\vec{v}, \dlambda) &= F(\vec{v}),& & \forall \vec{v} \in \Hdcnot, \label{contgeneralizedNewtoniterationweakform1}  \\
b(\ddirector, \gamma) &= G(\gamma),& & \forall \gamma \in \Ltwo{\Omega}, \label{contgeneralizedNewtoniterationweakform2} 
\end{align}
where $a(\cdot, \cdot)$ is a symmetric bilinear form, $b(\cdot, \cdot)$ is a bilinear form, and $F$ and $G$ are linear functionals. For simplicity, throughout this section, we drop the notation of $\ddirector$, $\dlambda$. Thus,
\begin{align}
a(\vec{u}, \vec{v}) = &K_1\Ltwoinner{\diverg \vec{u}}{\diverg \vec{v}}{\Omega} + K_3 \Ltwoinnerndim{\vec{Z}(\kdirector) \curl \vec{u}}{\curl \vec{v}}{\Omega}{3} \nonumber \\
& \qquad + (K_2-K_3) \Big(\Ltwoinner{\vec{u} \cdot \curl \vec{v}}{\kdirector \cdot \curl \kdirector}{\Omega}+\Ltwoinner{\kdirector \cdot \curl \vec{v}}{\vec{u} \cdot \curl \kdirector}{\Omega} \nonumber \\
&\qquad + \Ltwoinner{\kdirector \cdot \curl \kdirector}{\vec{v} \cdot \curl \vec{u}}{\Omega} + \Ltwoinner{\kdirector \cdot \curl \vec{u}}{\vec{v} \cdot \curl \kdirector}{\Omega} \nonumber \\
& \qquad + \Ltwoinner{\vec{u} \cdot \curl \kdirector}{\vec{v} \cdot \curl \kdirector}{\Omega}\Big) + \int_{\Omega} \klambda \ltwoinner{\vec{u}}{\vec{v}} \diff{V} \label{auvform},
\end{align}
and
\begin{align*} 
 b(\vec{v}, \gamma) = \int_{\Omega} \gamma \ltwoinner{\kdirector}{\vec{v}} \diff{V}.
\end{align*}
Moreover,
\begin{align*}
F(\vec{v}) &= - \Big(K_1\Ltwoinner{\diverg \kdirector}{\diverg \vec{v}}{\Omega} + K_3\Ltwoinnerndim{\vec{Z}(\kdirector) \curl \kdirector}{\curl \vec{v}}{\Omega}{3} \nonumber \\
&\qquad + (K_2-K_3)\Ltwoinner{\kdirector \cdot \curl \kdirector}{\vec{v} \cdot \curl \kdirector}{\Omega}  + \int_{\Omega} \klambda \ltwoinner{\kdirector}{\vec{v}} \diff{V}  \Big),
\end{align*}
and
\begin{align*}
G(\gamma) = -\frac{1}{2} \int_{\Omega} \gamma (\ltwoinner{\kdirector}{\kdirector} -1) \diff{V}.
\end{align*}

In this section, we aim to show that the system in \eqref{newtonhessian} is well-posed. Therefore, continuity, coercivity, and weak coercivity results are desired for the bilinear forms $a(\cdot, \cdot)$ and $b(\cdot, \cdot)$. Due to the complexity of the bilinear forms, deriving theoretical results in the continuum is challenging. However, the following lemmas hold.
\begin{lemma} \label{contcoercivityauv}
Under Assumption \ref{secass} and the assumption that $\klambda$ is pointwise non-negative,  if $\kappa =1$, there exists an $\alpha_0 >0$ such that $\alpha_0 \Hdcnorm{\vec{v}}{\Omega}^2 \leq a(\vec{v}, \vec{v})$ for all $\vec{v} \in \Hdcnot$.
\end{lemma}
\begin{proof}
The proof of this lemma is identical to that of Lemma \ref{coercivityauv} below.
\end{proof} \\
If additional regularity is asserted, such that $\ddirector$ and $\vec{v}$ are elements of $\Hdcone = \{ \vec{w} \in \Hdcnot : \curl \vec{w} \in \Hone{\Omega}^3 \}$ with norm $\Hdconenorm{\vec{w}}{\Omega}^2 = \Ltwonormndim{\vec{w}}{\Omega}{3}^2 + \Ltwonorm{\diverg \vec{w}}{\Omega}^2 + \Honenorm{\curl \vec{w}}{\Omega}^2$, where $\Honenorm{\cdot}{\Omega}$ denotes the standard norm on $\Hone{\Omega}$, then the next two lemmas hold for arbitrary $\kappa$.
\begin{lemma} \label{contboundedlinearforms}
Under Assumption \ref{secass}, $F$ and $G$ are bounded linear functionals on $\Hdcone$ and $\Ltwo{\Omega}$, respectively.
\end{lemma}
\begin{proof}
A simple application of the Cauchy-Schwarz inequality shows that $G(\gamma)$ is a bounded linear functional. 

For $F(\vec{v})$, observe that
\begin{align}
\vert F(\vec{v}) \vert &\leq K_1 \vert \Ltwoinner{\diverg \kdirector}{\diverg \vec{v}}{\Omega} \vert +  K_3 \vert \Ltwoinnerndim{\vec{A}(\kdirector) \curl \kdirector}{\curl \vec{v}}{\Omega}{3} \vert \nonumber \\
&\qquad + \vert K_2-K_3 \vert \vert \Ltwoinner{\kdirector \cdot \curl \kdirector}{\vec{v} \cdot \curl \kdirector}{\Omega} \vert  + \vert\int_{\Omega} \klambda \ltwoinner{\kdirector}{\vec{v}} \diff{V} \vert \label{Ftriangleinequalcont},
\end{align}
by the triangle inequality. Applying Cauchy-Schwarz inequalities to \eqref{Ftriangleinequalcont}, one obtains
\begin{align}
\vert F(\vec{v}) \vert & \leq K_1 \Ltwonorm{\diverg \kdirector}{\Omega} \Ltwonorm{\diverg \vec{v}}{\Omega} + K_3 \Ltwonormndim{\vec{A}(\kdirector) \curl \kdirector}{\Omega}{3} \Ltwonormndim{\curl \vec{v}}{\Omega}{3} \nonumber \\
&\qquad + \vert K_2-K_3 \vert \vert \Ltwoinner{\kdirector \cdot \curl \kdirector}{\vec{v} \cdot \curl \kdirector}{\Omega} \vert  + \Ltwonormndim{\klambda \kdirector}{\Omega}{3}\Ltwonormndim{\vec{v}}{\Omega}{3} \nonumber \\
&\leq K_1 \Ltwonorm{\diverg \kdirector}{\Omega} \Hdconenorm{\vec{v}}{\Omega} + K_3 \Ltwonormndim{\vec{A}(\kdirector) \curl \kdirector}{\Omega}{3} \Hdconenorm{\vec{v}}{\Omega} \nonumber \\
&\qquad + \vert K_2-K_3 \vert \vert \Ltwoinner{\kdirector \cdot \curl \kdirector}{\vec{v} \cdot \curl \kdirector}{\Omega} \vert  + \Ltwonormndim{\klambda \kdirector}{\Omega}{3}\Hdconenorm{\vec{v}}{\Omega}. \label{partialFinequalitycont}
\end{align}
In order to bound $\vert F(\vec{v}) \vert$, consider the final three summands separately. Note that since $\ltwonorm{\vec{A}(\kdirector)} \leq \Lambda$, where $\Lambda$ is the relevant upper bound from Lemma \ref{USPDlemma}, it is evident that 
\begin{equation}
\Ltwonormndim{\vec{A}(\kdirector)\curl \kdirector}{\Omega}{3} \leq \Lambda \Ltwonormndim{\curl \kdirector}{\Omega}{3}, \label{matrixcurltermcont}
\end{equation}
and that
\begin{align} \label{klambkdirtermcont}
\Ltwonormndim{\klambda \kdirector}{\Omega}{3}^2 &\leq \beta \int_{\Omega} \klambda^2 \diff{V} = C_{\lambda \director}^2,
\end{align}
where $\beta$ is the upper bound for \eqref{limitsonnlength}.
Finally, consider
\begin{equation*}
\vert \Ltwoinner{\kdirector \cdot \curl \kdirector}{\vec{v} \cdot \curl \kdirector}{\Omega} \vert = \vert \Ltwoinnerndim{(\kdirector \cdot \curl \kdirector) \curl \kdirector}{\vec{v}}{\Omega}{3} \vert.
\end{equation*}
Applying the Cauchy-Schwarz inequality,
\begin{align}
 \vert \Ltwoinnerndim{(\kdirector \cdot \curl \kdirector) \curl \kdirector}{\vec{v}}{\Omega}{3} \vert &\leq \Ltwonormndim{(\kdirector \cdot \curl \kdirector) \curl \kdirector}{\Omega}{3} \Ltwonormndim{\vec{v}}{\Omega}{3} \nonumber \\
 &\leq \Ltwonormndim{(\kdirector \cdot \curl \kdirector) \curl \kdirector}{\Omega}{3} \Hdconenorm{\vec{v}}{\Omega} \label{lastboundedfunctionalsummandcont}.
\end{align}
By Corollary 1.1 in \cite{Girault1}, since $\curl \kdirector \in \Hone{\Omega}^3$, $\curl \kdirector \cdot \curl \kdirector \in \Ltwo{\Omega}$. Note that
\begin{align*}
(\kdirector \cdot \curl \kdirector) \curl \kdirector \cdot (\kdirector \cdot \curl \kdirector) \curl \kdirector &= (\kdirector \cdot \curl \kdirector)^2  (\curl \kdirector \cdot \curl \kdirector) \nonumber \\
&\leq (\vert \kdirector \vert \cdot \vert \curl \kdirector \vert)^2 (\curl \kdirector \cdot \curl \kdirector) \\
&\leq \beta \cdot (\curl \kdirector \cdot \curl \kdirector)^2.
\end{align*}
Employing this in \eqref{lastboundedfunctionalsummandcont} and letting $\Ltwonorm{\curl \kdirector \cdot \curl \kdirector}{\Omega} = C_{\director \director}$,
\begin{align}
\Ltwonormndim{(\kdirector \cdot \curl \kdirector) \curl \kdirector}{\Omega}{3} &\leq \sqrt{\beta} \big( \int_{\Omega}(\curl \kdirector \cdot \curl \kdirector)^2 \diff{V} \big)^{1/2} \nonumber \\
&\leq \sqrt{\beta} C_{\director \director} \label{regularcurlcurltermcont}.
\end{align}
Therefore, using \eqref{partialFinequalitycont}-\eqref{klambkdirtermcont}, and \eqref{regularcurlcurltermcont},
\begin{align*}
\vert F(\vec{v}) \vert \leq& K_1 \Ltwonorm{\diverg \kdirector}{\Omega} \Hdconenorm{\vec{v}}{\Omega} + K_3 \Lambda \Ltwonormndim{\curl \kdirector}{\Omega}{3}\Hdconenorm{\vec{v}}{\Omega} \\ \nonumber
& + \vert K_2-K_3 \vert \sqrt{\beta} C_{\director \director}\Hdconenorm{\vec{v}}{\Omega} + C_{\lambda \director}\Hdconenorm{\vec{v}}{\Omega}.
\end{align*}
\end{proof}

\begin{lemma} \label{contbilinearformcontinuity}
Under Assumption \ref{secass}, $a(\vec{u}, \vec{v})$ and $b(\vec{v}, \gamma)$ are continuous for the norms $\Hdconenorm{\cdot}{\Omega}$ and $\Ltwonorm{\cdot}{\Omega}$.
\end{lemma}
\begin{proof}
First consider
\begin{align*}
\vert b(\vec{v}, \gamma) \vert &= \vert \int_{\Omega} \gamma \ltwoinner{\vec{v}}{\kdirector} \diff{V} \vert \nonumber \\
&\leq \Ltwonorm{\gamma}{\Omega} \Ltwonorm{\vec{v} \cdot \kdirector}{\Omega} \nonumber \\
&\leq \Ltwonorm{\gamma}{\Omega} \sqrt{\beta} \Ltwonorm{\vec{v}}{\Omega},
\end{align*}
by H\"{o}lder's inequality and \eqref{limitsonnlength}. Therefore, $b(\vec{v}, \gamma)$ is a continuous bilinear form.

For the continuity of $a(\vec{u}, \vec{v})$, observe that
\begin{align}
\vert a(\vec{u}, \vec{v}) \vert \leq & K_1 \vert \Ltwoinner{\diverg \vec{u}}{\diverg \vec{v}}{\Omega} \vert + K_3 \vert \Ltwoinnerndim{\vec{A}(\kdirector) \curl \vec{u}}{\curl \vec{v}}{\Omega}{3} \vert \nonumber \\
& + \vert K_2-K_3 \vert \Big( \vert \Ltwoinner{\vec{u} \cdot \curl \vec{v}}{\kdirector \cdot \curl \kdirector}{\Omega} \vert+\vert\Ltwoinner{\kdirector \cdot \curl \vec{v}}{\vec{u} \cdot \curl \kdirector}{\Omega} \vert \nonumber \\
& + \vert \Ltwoinner{\kdirector \cdot \curl \kdirector}{\vec{v} \cdot \curl \vec{u}}{\Omega} \vert + \vert \Ltwoinner{\kdirector \cdot \curl \vec{u}}{\vec{v} \cdot \curl \kdirector}{\Omega} \vert \nonumber \\
& + \vert \Ltwoinner{\vec{u} \cdot \curl \kdirector}{\vec{v} \cdot \curl \kdirector}{\Omega} \vert \Big) + \vert \int_{\Omega} \klambda \ltwoinner{\vec{u}}{\vec{v}} \diff{V} \vert \label{triangleauvinequalitycont},
\end{align}
by the triangle inequality. For simplicity, consider the components of the sum above. Note that
\begin{equation} 
\vert \Ltwoinner{\diverg \vec{u}}{\diverg \vec{v}}{\Omega} \vert \leq \Ltwonorm{\diverg \vec{u}}{\Omega} \Ltwonorm{\diverg \vec{v}}{\Omega} \leq \Hdconenorm{ \vec{u}}{\Omega} \Hdconenorm{ \vec{v}}{\Omega}. \label{divergenceauvcontineqcont}
\end{equation}
Considering $\vert \Ltwoinnerndim{\vec{A}(\kdirector) \curl \vec{u}}{\curl \vec{v}}{\Omega}{3} \vert$, using \eqref{matrixcurltermcont} implies that
\begin{align}
\vert \Ltwoinnerndim{\vec{A}(\kdirector) \curl \vec{u}}{\curl \vec{v}}{\Omega}{3} \vert &\leq \Ltwonormndim{\curl \vec{v}}{\Omega}{3} \Ltwonormndim{\vec{A}(\kdirector)\curl \vec{u}}{\Omega}{3} \nonumber \\
&\leq \Lambda \Hdconenorm{\vec{v}}{\Omega}  \Ltwonormndim{\curl \vec{u}}{\Omega}{3} \nonumber \\
& \leq \Lambda \Hdconenorm{\vec{v}}{\Omega}  \Hdconenorm{\vec{u}}{\Omega}. \label{Amatrixauvcontineqcont}
\end{align}

From the imbedding in Lemma 2.5 of \cite{Girault1}, if $\Omega$ is a convex polyhedron or has a $C^{1,1}$ boundary, then for any $\vec{w} \in \Hdivnot{\Omega} \cap \Hcurlnot{\Omega}$ there exists a $C_{\text{imb}}>0$ such that
\begin{equation*}
\Honenorm{\vec{w}}{\Omega}^2 \leq C_{\text{imb}} \Hdconenorm{\vec{w}}{\Omega}^2.
\end{equation*}
Furthermore, $\vec{w} \in \Honenot{\Omega}^3$ by \cite[Lemma 2.5]{Girault1}. Consider $\vert \Ltwoinner{\vec{u} \cdot \curl \vec{v}}{\kdirector \cdot \curl \kdirector}{\Omega} \vert$ from \eqref{triangleauvinequalitycont}. By Corollary 1.1 in \cite{Girault1}, the map $\vec{u} \cdot \curl \vec{v}$ is a \emph{continuous} bilinear map, $\Hone{\Omega}^3 \times \Hone{\Omega}^3 \to \Ltwo{\Omega}$. Therefore, there exists a $C_{\text{con}} >0$ such that
\begin{equation*}
\Ltwonorm{\vec{u} \cdot \curl \vec{v}}{\Omega} \leq C_{\text{con}} \Honenorm{\vec{u}}{\Omega} \Honenorm{\curl \vec{v}}{\Omega}.
\end{equation*}
By the Cauchy-Schwarz inequality
\begin{equation}
\vert \Ltwoinner{\vec{u} \cdot \curl \vec{v}}{\kdirector \cdot \curl \kdirector}{\Omega} \vert \leq \Ltwonorm{\vec{u} \cdot \curl \vec{v}}{\Omega} \Ltwonorm{\kdirector \cdot \curl \kdirector}{\Omega} \label{csk3kappafirstcont}.
\end{equation}
Let $C' = C_{\text{con}} C_{\text{imb}}$ and note that
\begin{align}
\Ltwonorm{\vec{u} \cdot \curl \vec{v}}{\Omega} &\leq C_{\text{con}} \Honenorm{\vec{u}}{\Omega} \Honenorm{\curl \vec{v}}{\Omega} \label{ucurlvcontinuitycont} \\
&\leq C' \Hdconenorm{\vec{u}}{\Omega} \Hdconenorm{\vec{v}}{\Omega} \label{ucurlvembiddingcont},
\end{align}
where \eqref{ucurlvcontinuitycont} is given by continuity and \eqref{ucurlvembiddingcont} is given by the imbedding. Hence,
\begin{align}
\vert \Ltwoinner{\vec{u} \cdot \curl \vec{v}}{\kdirector \cdot \curl \kdirector}{\Omega} \vert &\leq C' \Hdconenorm{\vec{u}}{\Omega} \Hdconenorm{\vec{v}}{\Omega} \Ltwonorm{\kdirector \cdot \curl \kdirector}{\Omega} \nonumber \\
& \leq C' \sqrt{\beta} \Ltwonormndim{\curl \kdirector}{\Omega}{3} \Hdconenorm{\vec{u}}{\Omega} \Hdconenorm{\vec{v}}{\Omega} \label{ucurlvinftyinequalcont}.
\end{align}

The next summand from \eqref{triangleauvinequalitycont} is
\begin{equation*} 
\vert\Ltwoinner{\kdirector \cdot \curl \vec{v}}{\vec{u} \cdot \curl \kdirector}{\Omega} \vert \leq \Ltwonorm{\kdirector \cdot \curl \vec{v}}{\Omega} \Ltwonorm{\vec{u} \cdot \curl \kdirector}{\Omega}.
\end{equation*}
Again bound
\begin{align*}
 \Ltwonorm{\kdirector \cdot \curl \vec{v}}{\Omega}& \leq \sqrt{\beta} \Hdconenorm{\vec{v}}{\Omega}.
\end{align*}
Since $\vec{u} \in \Honenot{\Omega}^3$ and $\curl \kdirector \in \Hone{\Omega}^3$,
\begin{align*}
\Ltwonorm{\vec{u} \cdot \curl \kdirector}{\Omega} &\leq C_{\text{con}}\Honenorm{\vec{u}}{\Omega} \Honenorm{\curl \kdirector}{\Omega} \nonumber \\
& \leq C' \Hdconenorm{\vec{u}}{\Omega}\Honenorm{\curl \kdirector}{\Omega}.
\end{align*}
Therefore,
\begin{equation} \label{secondk3kappacontineqcont}
\vert\Ltwoinner{\kdirector \cdot \curl \vec{v}}{\vec{u} \cdot \curl \kdirector}{\Omega} \vert \leq \sqrt{\beta} C' \Honenorm{\curl \kdirector}{\Omega} \Hdconenorm{\vec{u}}{\Omega}\Hdconenorm{\vec{v}}{\Omega}.
\end{equation}
Now consider $\vert \Ltwoinner{\kdirector \cdot \curl \kdirector}{\vec{v} \cdot \curl \vec{u}}{\Omega} \vert$ and note that this inner product is the same as that in \eqref{csk3kappafirstcont} with the roles of $\vec{u}$ and $\vec{v}$ reversed. Since $\vec{u}$ and $\vec{v}$ are from the same space, the steps for deriving \eqref{ucurlvinftyinequalcont} are equally valid. Thus,
\begin{equation}
\vert \Ltwoinner{\kdirector \cdot \curl \kdirector}{\vec{v} \cdot \curl \vec{u}}{\Omega} \vert\leq C' \sqrt{\beta} \Ltwonormndim{\curl \kdirector}{\Omega}{3} \Hdconenorm{\vec{u}}{\Omega} \Hdconenorm{\vec{v}}{\Omega} \label{thirdk3kappacontineqcont}.
\end{equation}
Similarly, the inequality for $\vert \Ltwoinner{\kdirector \cdot \curl \vec{u}}{\vec{v} \cdot \curl \kdirector}{\Omega} \vert$ is derived in an analogous manner to that of \eqref{secondk3kappacontineqcont}. Thus,
\begin{equation}
\vert \Ltwoinner{\kdirector \cdot \curl \vec{u}}{\vec{v} \cdot \curl \kdirector}{\Omega} \vert \leq \sqrt{\beta} C' \Honenorm{\curl \kdirector}{\Omega} \Hdconenorm{\vec{u}}{\Omega}\Hdconenorm{\vec{v}}{\Omega} \label{forthk3kappacontineqcont}.
\end{equation}

Next, examine
\begin{equation*}
\vert \Ltwoinner{\vec{u} \cdot \curl \kdirector}{\vec{v} \cdot \curl \kdirector}{\Omega} \vert \leq \Ltwonorm{\vec{u} \cdot \curl \kdirector}{\Omega} \Ltwonorm{\vec{v} \cdot \curl \kdirector}{\Omega}.
\end{equation*}
Since $\curl \kdirector \in \Hone{\Omega}^3$,
\begin{align*}
\Ltwonorm{\vec{u} \cdot \curl \kdirector}{\Omega} \leq C_{\text{con}} \Honenorm{\vec{u}}{\Omega} \Honenorm{\curl \kdirector}{\Omega} \leq C' \Hdconenorm{\vec{u}}{\Omega} \Honenorm{\curl \kdirector}{\Omega} ,\\
 \Ltwonorm{\vec{v} \cdot \curl \kdirector}{\Omega} \leq C_{\text{con}} \Honenorm{\vec{v}}{\Omega} \Honenorm{\curl \kdirector}{\Omega} \leq C' \Hdconenorm{\vec{v}}{\Omega} \Honenorm{\curl \kdirector}{\Omega}.
\end{align*}
Thus,
\begin{equation}
\vert \Ltwoinner{\vec{u} \cdot \curl \kdirector}{\vec{v} \cdot \curl \kdirector}{\Omega} \vert \leq (C')^2 \Honenorm{\curl \kdirector}{\Omega}^2 \Hdconenorm{\vec{u}}{\Omega}\Hdconenorm{\vec{v}}{\Omega} \label{k3kappafiftcontineqcont}.
\end{equation}
Finally,
\begin{align}
\vert \int_{\Omega} \klambda \ltwoinner{\vec{u}}{\vec{v}} \diff{V} \vert &\leq \Ltwonorm{\klambda}{\Omega} \Ltwonorm{\vec{u} \cdot \vec{v}}{\Omega} \nonumber \\
&\leq \Ltwonorm{\klambda}{\Omega} C_{\text{con}} \Honenorm{\vec{u}}{\Omega} \Honenorm{\vec{v}}{\Omega}\nonumber \\
& \leq \Ltwonorm{\klambda}{\Omega} C'C_{\text{imb}} \Hdconenorm{\vec{u}}{\Omega} \Hdconenorm{\vec{v}}{\Omega} \label{lambdaimbedineqcont}.
\end{align}
Combining \eqref{divergenceauvcontineqcont}, \eqref{Amatrixauvcontineqcont}, and \eqref{ucurlvinftyinequal}-\eqref{lambdaimbedineqcont},
\begin{align*}
a(\vec{u},\vec{v}) &\leq \Big( K_1 + K_3 \Lambda + \vert K_2-K_3 \vert \big( 2C' \sqrt{\beta} \Ltwonormndim{\curl \kdirector}{\Omega}{3}+ 2\sqrt{\beta} C' \Honenorm{\curl \kdirector}{\Omega} \nonumber \\
&\qquad + (C')^2 \Honenorm{\curl \kdirector}{\Omega}^2\big) + \Ltwonorm{\klambda}{\Omega} C'C_{\text{imb}} \Big) \Hdconenorm{\vec{u}}{\Omega} \Hdconenorm{\vec{v}}{\Omega}.
\end{align*}
\end{proof}

The auxiliary regularity above poses a number of theoretical problems. For the well-posedness of the continuum system, coercivity and weak coercivity must be shown in the more intricate $\Hdconespace$ norm. Moreover, conforming finite elements for this space, such as Bogner-Fox-Schmit elements \cite{Bogner1}, are undesirably cumbersome and present notable difficulties in demonstrating stability for this linearization system. However, in the discrete setting, results guaranteeing the existence and uniqueness of solutions to the discrete Newton systems at each step are attained under less strict assumptions.

\subsection{Discrete System Preliminaries}

Performing the outlined Newton iterations for free elastic effects necessitates solving the above Newton systems for the update functions $\ddirector$ and $\dlambda$. Thus, finite elements are used to numerically approximate the updates. Finite dimensional spaces $V_h \subset \Hdcnot$ and $\Pi_h \subset \Ltwo{\Omega}$ are considered, yielding the discrete variational problem
\begin{align}
\vec{a}(\ddirector_h, \vec{v}_h) + b(\vec{v}_h, \dlambda_h) &= \vec{F}(\vec{v}_h),& &\forall \vec{v}_h \in V_h, \label{generalizedNewtoniterationweakform1} \\
b(\ddirector_h, \gamma_h) &= G(\gamma_h),& &\forall \gamma_h \in \Pi_h. \label{generalizedNewtoniterationweakform2}
\end{align}
Note that Assumption \ref{secass} implies that $\ddirector_h$ and $\vec{v}_h$ are also elements of $\Hdivnot{\Omega} \cap \Hcurlnot{\Omega}$. Throughout the rest of this section, the developed theory applies exclusively to discrete spaces. Therefore, except when necessary for clarity, we drop the subscript $h$ along with the notation $\ddirector$ and $\dlambda$. For instance, we write $a(\vec{u}, \vec{v})$ to indicate the bilinear form in \eqref{auvform} operating on the discrete space $V_h \times V_h$. 

The existence and uniqueness theory in the following subsections is explicitly developed in the presence of full Dirichlet boundary conditions. However, the theory is equally applicable for a rectangular domain with mixed Dirichlet and periodic boundary conditions. Such a domain is considered in the numerical experiments presented herein.

Let $\{\triangulation \}$, $0< h \leq 1$, be a family of quadrilateral subdivisions of $\Omega$, such that
\begin{equation} 
\max \{\diam T : T \in \triangulation \} \leq h\ \diam \Omega. \label{maxuniformity}
\end{equation}
Further, assume that $\{\triangulation \}$ is quasi-uniform so that there exists a $\rho > 0$, such that
\begin{equation} \label{quasiuniform}
\min \{\diam B_T : T \in \triangulation \} \geq \rho\ h\ \diam \Omega,
\end{equation}
for all $h \in (0, 1]$, where $B_T$ is the largest ball contained in $T$, such that $T$ is star-shaped with respect to $B_T$ \cite{Brenner1}. Denote the measure of $T \in \triangulation$ as $\vert T \vert$. Furthermore, let $Q_p$ denote piecewise $C^0$ polynomials of degree $p \geq 1$ on $\triangulation$ and $P_0$ denote the space of piecewise constants on $\triangulation$. Next, define a bubble space 
\begin{equation*}
V_h^b = \{\vec{v} \in C_c(\Omega)^3 : \vec{v}|_{T} =a_T b_T \kdirector |_T, \forall T \in \triangulation \},
\end{equation*}
where $C_c(\Omega)$ denotes the space of compactly supported continuous functions on $\Omega$, $b_T$ is the quadratic bubble function \cite{Mourad1} that vanishes on $\partial T \in \triangulation$, and $a_T$ is a constant coefficient associated with $b_T$. The bubble functions are constructed \cite{Pierre1}, such that
\begin{align} 
\int_T b_T \diff{V} &= 1,& &\forall T \in \triangulation, \label{bubbleprop1} \\
b_T &> 0,& &\forall \vec{x} \in T \label{bubbleprop2}.
\end{align}
Then, we consider the pair of spaces
\begin{align}
\Pi_h &= P_0, \label{pispace} \\
V_h &= \{ \vec{v} \in Q_m \times Q_m \times Q_m \oplus V_h^b : \vec{v} = \vec{0} \text{ on } \partial \Omega \}. \label{vspace}
\end{align}

In the following sections, to demonstrate the existence and uniqueness of solutions to the system given by \eqref{generalizedNewtoniterationweakform1} and \eqref{generalizedNewtoniterationweakform2}, we show that $a(\vec{u}, \vec{v})$ is a coercive and continuous bilinear form and that $b(\vec{v}, \gamma)$ is a continuous and weakly coercive bilinear form \cite{Brenner1, Boffi1, Braess1, Babuska1} for the above spaces, $V_h$ and $\Pi_h$. Throughout the remainder of this section, we further assume that $\kdirector \in Q_p$, for some $p \geq 1$, so that $V_h \subset Q_l \times Q_l \times Q_l$ for $l = \max (m, p+2)$.

\subsection{Discrete Continuity}

In this section, we show that the right hand sides of \eqref{generalizedNewtoniterationweakform1} and \eqref{generalizedNewtoniterationweakform2} are continuous linear functionals and that the bilinear forms $a(\vec{u}, \vec{v})$ and $b(\vec{v}, \gamma)$ are continuous for the assumptions discussed above.
\begin{lemma} \label{boundedlinearforms}
Under Assumption \ref{secass}, $F$ and $G$ are bounded linear functionals on $V_h$ and $\Pi_h$, respectively.
\end{lemma}
\begin{proof}
A simple application of the Cauchy-Schwarz inequality shows that $G(\gamma)$ is a bounded linear functional. 

For $F(\vec{v})$, observe that
\begin{align}
\vert F(\vec{v}) \vert &\leq  K_1 \vert \Ltwoinner{\diverg \kdirector}{\diverg \vec{v}}{\Omega} \vert +  K_3 \vert \Ltwoinnerndim{\vec{Z}(\kdirector) \curl \kdirector}{\curl \vec{v}}{\Omega}{3} \vert \nonumber \\
&\qquad + \vert K_2-K_3 \vert \vert \Ltwoinner{\kdirector \cdot \curl \kdirector}{\vec{v} \cdot \curl \kdirector}{\Omega} \vert  + \left \vert \int_{\Omega} \klambda \ltwoinner{\kdirector}{\vec{v}} \diff{V} \right \vert \label{Ftriangleinequal},
\end{align}
by the triangle inequality. Applying Cauchy-Schwarz inequalities to \eqref{Ftriangleinequal}, one obtains
\begin{align}
\vert F(\vec{v}) \vert & \leq K_1 \Ltwonorm{\diverg \kdirector}{\Omega} \Ltwonorm{\diverg \vec{v}}{\Omega} +  K_3 \Ltwonormndim{\vec{Z}(\kdirector) \curl \kdirector}{\Omega}{3} \Ltwonormndim{\curl \vec{v}}{\Omega}{3} \nonumber \\
&\qquad + \vert K_2-K_3 \vert \vert \Ltwoinner{\kdirector \cdot \curl \kdirector}{\vec{v} \cdot \curl \kdirector}{\Omega} \vert  + \Ltwonormndim{\klambda \kdirector}{\Omega}{3}\Ltwonormndim{\vec{v}}{\Omega}{3} \nonumber \\
&\leq K_1 \Ltwonorm{\diverg \kdirector}{\Omega} \Hdcnorm{\vec{v}}{\Omega} +  K_3 \Ltwonormndim{\vec{Z}(\kdirector) \curl \kdirector}{\Omega}{3} \Hdcnorm{\vec{v}}{\Omega} \nonumber \\
&\qquad + \vert K_2-K_3 \vert \vert \Ltwoinner{\kdirector \cdot \curl \kdirector}{\vec{v} \cdot \curl \kdirector}{\Omega} \vert  + \Ltwonormndim{\klambda \kdirector}{\Omega}{3}\Hdcnorm{\vec{v}}{\Omega}. \label{partialFinequality}
\end{align}
In order to bound $\vert F(\vec{v}) \vert$, consider the final three summands separately. Note that since $\ltwonorm{\vec{Z}(\kdirector)} \leq \Lambda$, where $\Lambda$ is the relevant upper bound from Lemma \ref{USPDlemma}, it is evident that 
\begin{equation}
\Ltwonormndim{\vec{Z}(\kdirector)\curl \kdirector}{\Omega}{3} \leq \Lambda \Ltwonormndim{\curl \kdirector}{\Omega}{3}, \label{matrixcurlterm}
\end{equation}
and that
\begin{align}
\Ltwonormndim{\klambda \kdirector}{\Omega}{3}^2 &\leq \beta \int_{\Omega} \klambda^2 \diff{V} = C_1^2 \label{klambkdirterm},
\end{align}
where $\beta$ is the upper bound in \eqref{limitsonnlength}.
Finally, consider
\begin{equation*}
\vert \Ltwoinner{\kdirector \cdot \curl \kdirector}{\vec{v} \cdot \curl \kdirector}{\Omega} \vert = \vert \Ltwoinnerndim{(\kdirector \cdot \curl \kdirector) \curl \kdirector}{\vec{v}}{\Omega}{3} \vert.
\end{equation*}
Applying the Cauchy-Schwarz inequality,
\begin{align}
 \vert \Ltwoinnerndim{(\kdirector \cdot \curl \kdirector) \curl \kdirector}{\vec{v}}{\Omega}{3} \vert &\leq \Ltwonormndim{(\kdirector \cdot \curl \kdirector) \curl \kdirector}{\Omega}{3} \Ltwonormndim{\vec{v}}{\Omega}{3} \nonumber \\
 &\leq \Ltwonormndim{(\kdirector \cdot \curl \kdirector) \curl \kdirector}{\Omega}{3} \Hdcnorm{\vec{v}}{\Omega} \label{lastboundedfunctionalsummand}.
\end{align}
Next, note that
\begin{align}
(\kdirector \cdot \curl \kdirector) \curl \kdirector \cdot (\kdirector \cdot \curl \kdirector) \curl \kdirector &= (\kdirector \cdot \curl \kdirector)^2  (\curl \kdirector \cdot \curl \kdirector) \nonumber \\
&\leq (\vert \kdirector \vert \cdot \vert \curl \kdirector \vert)^2 (\curl \kdirector \cdot \curl \kdirector) \nonumber \\
&\leq \beta \cdot (\curl \kdirector \cdot \curl \kdirector)^2 \label{curlcurlinequal}.
\end{align}
Furthermore, $\curl \kdirector$ is a vector of piecewise polynomials. Therefore, $\curl \kdirector \cdot \curl \kdirector \in \Ltwo{\Omega}$. Employing \eqref{curlcurlinequal} and letting $\Ltwonorm{\curl \kdirector \cdot \curl \kdirector}{\Omega} = C_2$,
\begin{align}
\Ltwonormndim{(\kdirector \cdot \curl \kdirector) \curl \kdirector}{\Omega}{3} &\leq \sqrt{\beta} \big( \int_{\Omega}(\curl \kdirector \cdot \curl \kdirector)^2 \diff{V} \big)^{1/2} \nonumber \\
&\leq \sqrt{\beta} C_2 \label{regularcurlcurlterm}.
\end{align}
Therefore, using \eqref{partialFinequality}-\eqref{lastboundedfunctionalsummand}, and \eqref{regularcurlcurlterm},
\begin{align*}
\vert F(\vec{v}) \vert \leq& K_1 \Ltwonorm{\diverg \kdirector}{\Omega} \Hdcnorm{\vec{v}}{\Omega} + K_3 \Lambda \Ltwonormndim{\curl \kdirector}{\Omega}{3}\Hdcnorm{\vec{v}}{\Omega} \\ \nonumber
& + \vert K_2-K_3 \vert \sqrt{\beta} C_2\Hdcnorm{\vec{v}}{\Omega} + C_1\Hdcnorm{\vec{v}}{\Omega},
\end{align*}
implying $F(\vec{v})$ is a bounded linear functional on $V_h$.
\end{proof}

\begin{lemma} \label{bilinearformcontinuity}
Under Assumption \ref{secass}, $a(\vec{u}, \vec{v})$ and $b(\vec{v}, \gamma)$ are continuous.
\end{lemma}
\begin{proof}
First consider
\begin{align*}
\vert b(\vec{v}, \gamma) \vert &= \left \vert \int_{\Omega} \gamma \ltwoinner{\vec{v}}{\kdirector} \diff{V} \right \vert \nonumber \\
&\leq \Ltwonorm{\gamma}{\Omega} \Ltwonorm{\vec{v} \cdot \kdirector}{\Omega} \nonumber \\
&\leq \Ltwonorm{\gamma}{\Omega} \sqrt{\beta} \Ltwonorm{\vec{v}}{\Omega},
\end{align*}
by H\"{o}lder's inequality and \eqref{limitsonnlength}. Therefore, $b(\vec{v}, \gamma)$ is a continuous bilinear form.

For the continuity of $a(\vec{u}, \vec{v})$, observe that
\begin{align}
\vert a(\vec{u}, \vec{v}) \vert \leq & K_1 \vert \Ltwoinner{\diverg \vec{u}}{\diverg \vec{v}}{\Omega} \vert + K_3 \vert \Ltwoinnerndim{\vec{Z}(\kdirector) \curl \vec{u}}{\curl \vec{v}}{\Omega}{3} \vert \nonumber \\
& + \vert K_2-K_3 \vert \Big( \vert \Ltwoinner{\vec{u} \cdot \curl \vec{v}}{\kdirector \cdot \curl \kdirector}{\Omega} \vert+\vert\Ltwoinner{\kdirector \cdot \curl \vec{v}}{\vec{u} \cdot \curl \kdirector}{\Omega} \vert \nonumber \\
& + \vert \Ltwoinner{\kdirector \cdot \curl \kdirector}{\vec{v} \cdot \curl \vec{u}}{\Omega} \vert + \vert \Ltwoinner{\kdirector \cdot \curl \vec{u}}{\vec{v} \cdot \curl \kdirector}{\Omega} \vert \nonumber \\
& + \vert \Ltwoinner{\vec{u} \cdot \curl \kdirector}{\vec{v} \cdot \curl \kdirector}{\Omega} \vert \Big) + \left \vert \int_{\Omega} \klambda \ltwoinner{\vec{u}}{\vec{v}} \diff{V} \right \vert \label{triangleauvinequality},
\end{align}
by the triangle inequality. For simplicity, consider the components of the sum above. Note that
\begin{equation} 
\vert \Ltwoinner{\diverg \vec{u}}{\diverg \vec{v}}{\Omega} \vert \leq \Ltwonorm{\diverg \vec{u}}{\Omega} \Ltwonorm{\diverg \vec{v}}{\Omega} \leq \Hdcnorm{ \vec{u}}{\Omega} \Hdcnorm{ \vec{v}}{\Omega}. \label{divergenceauvcontineq}
\end{equation}
Considering $\vert \Ltwoinnerndim{\vec{Z}(\kdirector) \curl \vec{u}}{\curl \vec{v}}{\Omega}{3} \vert$, using \eqref{matrixcurlterm} implies that
\begin{align}
\vert \Ltwoinnerndim{\vec{Z}(\kdirector) \curl \vec{u}}{\curl \vec{v}}{\Omega}{3} \vert &\leq \Ltwonormndim{\curl \vec{v}}{\Omega}{3} \Ltwonormndim{\vec{Z}(\kdirector)\curl \vec{u}}{\Omega}{3} \nonumber \\
&\leq \Lambda \Hdcnorm{\vec{v}}{\Omega}  \Ltwonormndim{\curl \vec{u}}{\Omega}{3} \nonumber \\
& \leq \Lambda \Hdcnorm{\vec{v}}{\Omega}  \Hdcnorm{\vec{u}}{\Omega}. \label{Amatrixauvcontineq}
\end{align}
By the Cauchy-Schwarz inequality,
\begin{align}
\vert \Ltwoinner{\vec{u} \cdot \curl \vec{v}}{\kdirector \cdot \curl \kdirector}{\Omega} \vert &= \vert \Ltwoinner{(\kdirector \cdot \curl \kdirector)\vec{u}}{\curl \vec{v}}{\Omega} \vert \nonumber \\
&\leq \Ltwonorm{(\kdirector \cdot \curl \kdirector)\vec{u}}{\Omega} \Ltwonorm{\curl \vec{v}}{\Omega} \label{csk3kappafirst}.
\end{align}
Note that
\begin{equation*}
(\kdirector \cdot \curl \kdirector)^2 \leq \ltwonorm{\kdirector}^2 \ltwonorm{\curl \kdirector}^2 \leq \beta \ltwonorm{\curl \kdirector}^2.
\end{equation*}
Furthermore, since $\curl \kdirector$ is a vector of piecewise polynomials, $\ltwonorm{\curl \kdirector}^2$ is bounded. Letting $\displaystyle{C_{\text{sup}} = \sup_{\vec{x} \in \Omega} \ltwonorm{\curl \kdirector}^2}$, 
\begin{align*}
\Ltwonormndim{(\kdirector \cdot \curl \kdirector)\vec{u}}{\Omega}{3} &= \left ( \int_{\Omega} (\kdirector \cdot \curl \kdirector)^2 (\vec{u} \cdot \vec{u}) \diff{V} \right)^{1/2} \nonumber \\
& \leq \sqrt{\beta} \left ( \int_{\Omega} \ltwonorm{\curl \kdirector}^2 (\vec{u} \cdot \vec{u}) \diff{V} \right)^{1/2} \nonumber \\
& \leq \sqrt{\beta C_{\text{sup}}} \Ltwonormndim{\vec{u}}{\Omega}{3}.
\end{align*}
Hence,
\begin{equation}
\vert \Ltwoinner{\vec{u} \cdot \curl \vec{v}}{\kdirector \cdot \curl \kdirector}{\Omega} \vert \leq \sqrt{\beta C_{\text{sup}}} \Hdcnorm{\vec{u}}{\Omega}\Hdcnorm{\vec{v}}{\Omega} \label{ucurlvinftyinequal}.
\end{equation}
The next summand from \eqref{triangleauvinequality} is
\begin{equation*} \label{thirdk3kappainner}
\vert\Ltwoinner{\kdirector \cdot \curl \vec{v}}{\vec{u} \cdot \curl \kdirector}{\Omega} \vert \leq \Ltwonorm{\kdirector \cdot \curl \vec{v}}{\Omega} \Ltwonorm{\vec{u} \cdot \curl \kdirector}{\Omega},
\end{equation*}
with
\begin{align*}
 \Ltwonorm{\kdirector \cdot \curl \vec{v}}{\Omega}& \leq \sqrt{\beta} \Hdcnorm{\vec{v}}{\Omega}.
\end{align*}
Furthermore,
\begin{equation*}
\Ltwonorm{\vec{u} \cdot \curl \kdirector}{\Omega} \leq \sqrt{C_{\sup}} \Ltwonormndim{\vec{u}}{\Omega}{3}.
\end{equation*}
Therefore,
\begin{equation} \label{secondk3kappacontineq}
\vert\Ltwoinner{\kdirector \cdot \curl \vec{v}}{\vec{u} \cdot \curl \kdirector}{\Omega} \vert \leq \sqrt{\beta C_{\text{sup}}} \Hdcnorm{\vec{v}}{\Omega} \Hdcnorm{\vec{u}}{\Omega}.
\end{equation}
Now consider $\vert \Ltwoinner{\kdirector \cdot \curl \kdirector}{\vec{v} \cdot \curl \vec{u}}{\Omega} \vert$ and note that this inner product is the same as that in \eqref{csk3kappafirst} with the roles of $\vec{u}$ and $\vec{v}$ reversed. Since $\vec{u}$ and $\vec{v}$ are from the same space, the steps for deriving \eqref{ucurlvinftyinequal} are equally valid. Thus,
\begin{equation}
\vert \Ltwoinner{\kdirector \cdot \curl \kdirector}{\vec{v} \cdot \curl \vec{u}}{\Omega} \vert \leq \sqrt{\beta C_{\text{sup}}} \Hdcnorm{\vec{u}}{\Omega}\Hdcnorm{\vec{v}}{\Omega} \label{thirdk3kappacontineq}.
\end{equation}
Similarly, the inequality for $\vert \Ltwoinner{\kdirector \cdot \curl \vec{u}}{\vec{v} \cdot \curl \kdirector}{\Omega} \vert$ is derived in an analogous manner to that of \eqref{secondk3kappacontineq}. Thus,
\begin{equation}
\vert \Ltwoinner{\kdirector \cdot \curl \vec{u}}{\vec{v} \cdot \curl \kdirector}{\Omega} \vert \leq \sqrt{\beta C_{\text{sup}}} \Hdcnorm{\vec{v}}{\Omega} \Hdcnorm{\vec{u}}{\Omega}. \label{forthk3kappacontineq}
\end{equation}
Next, examine
\begin{equation*}
\vert \Ltwoinner{\vec{u} \cdot \curl \kdirector}{\vec{v} \cdot \curl \kdirector}{\Omega} \vert \leq \Ltwonorm{\vec{u} \cdot \curl \kdirector}{\Omega} \Ltwonorm{\vec{v} \cdot \curl \kdirector}{\Omega}.
\end{equation*}
Since $\curl \kdirector$ is a vector of piecewise polynomials,
\begin{align*}
\Ltwonorm{\vec{u} \cdot \curl \kdirector}{\Omega} \leq \sqrt{C_{\sup}} \Ltwonormndim{\vec{u}}{\Omega}{3},\\
 \Ltwonorm{\vec{v} \cdot \curl \kdirector}{\Omega} \leq \sqrt{C_{\sup}} \Ltwonormndim{\vec{v}}{\Omega}{3}.
\end{align*}
Thus,
\begin{equation}
\vert \Ltwoinner{\vec{u} \cdot \curl \kdirector}{\vec{v} \cdot \curl \kdirector}{\Omega} \vert \leq C_{\text{sup}} \Hdcnorm{\vec{u}}{\Omega}\Hdcnorm{\vec{v}}{\Omega} \label{k3kappafiftcontineq}.
\end{equation}
Finally, since $\klambda$ is piecewise constant, $\klambda^2$ is bounded. Letting $\displaystyle{C_{\lambda} = \sup_{\vec{x} \in \Omega} \klambda^2}$,
\begin{align}
\left \vert \int_{\Omega} \klambda \ltwoinner{\vec{u}}{\vec{v}} \diff{V} \right \vert &\leq \Ltwonormndim{\klambda \vec{u}}{\Omega}{3} \Ltwonorm{\vec{v}}{\Omega} \nonumber \\
& \leq \sqrt{C_{\lambda}} \Ltwonormndim{\vec{u}}{\Omega}{3} \Hdcnorm{\vec{v}}{\Omega} \nonumber \\
&\leq \sqrt{C_{\lambda}} \Hdcnorm{\vec{u}}{\Omega} \Hdcnorm{\vec{v}}{\Omega} \label{lambdaimbedineq}.
\end{align}
Combining \eqref{divergenceauvcontineq}, \eqref{Amatrixauvcontineq}, and \eqref{ucurlvinftyinequal}-\eqref{lambdaimbedineq},
\begin{align*}
a(\vec{u},\vec{v}) &\leq \Big( K_1 + K_3 \Lambda + \vert K_2-K_3 \vert \big( 4 \sqrt{\beta C_{\text{sup}}} + C_{\text{sup}} \big) + \sqrt{C_{\lambda}} \Big) \Hdcnorm{\vec{u}}{\Omega} \Hdcnorm{\vec{v}}{\Omega}.
\end{align*}
\end{proof}

\subsection{Discrete Coercivity}

In this section, two proofs of the coercivity of $a(\vec{u}, \vec{v})$ are given. The first is for the case when $\kappa=1$. The second addresses coercivity when $\kappa$ lies in a neighborhood of unity. For both proofs, we use the additional assumption that the approximation is close enough to the solution such that the Lagrange multiplier, $\klambda$, is pointwise non-negative. This assumption is reasonable since at the solution, $\director_{*}$, $\lambda_{*}$ may be chosen arbitrarily.
\begin{lemma} \label{coercivityauv}
Under Assumption \ref{secass} and the assumption that $\klambda$ is pointwise non-negative,  if $\kappa =1$, there exists an $\alpha_0 >0$ such that $\alpha_0 \Hdcnorm{\vec{v}}{\Omega}^2 \leq a(\vec{v}, \vec{v})$ for all $\vec{v} \in V_h$.
\end{lemma}
\begin{proof}
Note that since $\kappa=1$, $(K_2-K_3) = 0$, and
\begin{align*}
a(\vec{v}, \vec{v}) = &K_1\Ltwoinner{\diverg \vec{v}}{\diverg \vec{v}}{\Omega} + K_3 \Ltwoinnerndim{\curl \vec{v}}{\curl \vec{v}}{\Omega}{3} + \int_{\Omega} \klambda \ltwoinner{\vec{v}}{\vec{v}} \diff{V}.
\end{align*}
Thus, it remains to show that there exists $\alpha_0 >0$ such that
\begin{align*}
\alpha_0 \Hdcnorm{\vec{v}}{\Omega}^2 \leq &K_1\Ltwoinner{\diverg \vec{v}}{\diverg \vec{v}}{\Omega} + K_3 \Ltwoinnerndim{\curl \vec{v}}{\curl \vec{v}}{\Omega}{3} + \int_{\Omega} \klambda \ltwoinner{\vec{v}}{\vec{v}} \diff{V}.
\end{align*}
From Remark 2.7 in \cite{Girault1}, there exists $C_3 >0$ such that
\begin{equation*}
\Ltwonormndim{\nabla \vec{v}}{\Omega}{3}^2 \leq C_3^2 \big(\Ltwonorm{\diverg \vec{v}}{\Omega}^2 + \Ltwonormndim{\curl \vec{v}}{\Omega}{3}^2 \big).
\end{equation*}
Moreover, recall that $\Ltwonormndim{\vec{v}}{\Omega}{3}^2 \leq C_4 \Ltwonormndim{\nabla \vec{v}}{\Omega}{3}^2$ by the classical Poincar\'{e}-Friedrichs' inequality. Hence, for $C = C_4C_3^2>0$,
\begin{equation} \label{PFineqforDCuse}
\Ltwonormndim{\vec{v}}{\Omega}{3}^2 \leq C \big(\Ltwonorm{\diverg \vec{v}}{\Omega}^2 + \Ltwonormndim{\curl \vec{v}}{\Omega}{3}^2 \big).
\end{equation}
Since $\Hdcnorm{\vec{v}}{\Omega}^2 = \Ltwonormndim{\vec{v}}{\Omega}{3}^2 + \Ltwonorm{\diverg \vec{v}}{\Omega}^2 + \Ltwonormndim{\curl \vec{v}}{\Omega}{3}^2$, then
\begin{equation*}
\Hdcnorm{\vec{v}}{\Omega}^2 \leq (C+1) \big(\Ltwonorm{\diverg \vec{v}}{\Omega}^2 + \Ltwonormndim{\curl \vec{v}}{\Omega}{3}^2 \big).
\end{equation*}
Letting $K = \min(K_1, K_3) >0$ and $\alpha_0 = K/(C+1)$, it follows that
\begin{equation}
\alpha_0 \Hdcnorm{\vec{v}}{\Omega}^2 \leq K \big(\Ltwonorm{\diverg \vec{v}}{\Omega}^2 + \Ltwonormndim{\curl \vec{v}}{\Omega}{3}^2 \big) \leq K_1 \Ltwonorm{\diverg \vec{v}}{\Omega}^2 + K_3\Ltwonormndim{\curl \vec{v}}{\Omega}{3}^2. \label{divcurlcoercivity}
\end{equation}
Finally, it was assumed that $\klambda$ is pointwise non-negative, implying
\begin{equation*}
\int_{\Omega} \klambda \ltwoinner{\vec{v}}{\vec{v}} \diff{V} \geq 0.
\end{equation*}
Therefore, \eqref{divcurlcoercivity} implies that
\begin{equation*}
\alpha_0 \Hdcnorm{\vec{v}}{\Omega}^2 \leq K_1\Ltwoinner{\diverg \vec{v}}{\diverg \vec{v}}{\Omega} + K_3 \Ltwoinnerndim{\curl \vec{v}}{\curl \vec{v}}{\Omega}{3} + \int_{\Omega} \klambda \ltwoinner{\vec{v}}{\vec{v}} \diff{V}.
\end{equation*}
\end{proof}

The assumption that $\kappa=1$ is a common modeling approach. In fact, this supposition represents a weaker constraint than is seen in the many models that utilize the one-constant approximation, cf. \cite{Ramage1, Liu1, Stewart1, Cohen1}. However, it is possible to loosen the restriction that $\kappa = 1$ and still maintain the coercivity of $a(\vec{u},\vec{v})$ with a small data type assumption on $\kappa$. That is, we assume that $\kappa$ varies within a certain, possibly small, range of unity. Small data assumptions are common, for instance, in the study of solutions to the Navier-Stokes' equations \cite{Fujita1, Leray1, Marusic-Paloka1}, where bounds are imposed on certain norms of the initial data in order to demonstrate existence and uniqueness of solutions.

\begin{lemma}[Small Data] \label{coercivitysmalldata}
Under Assumption \ref{secass} and the assumption that $\lambda_k$ is pointwise non-negative, there exists $\epsilon_1, \epsilon_2 > 0$, dependent on $\beta=\max \ltwonorm{\director}^2$, such that if $\kappa \in (1-\epsilon_2, 1+\epsilon_1)$, then $a(\vec{u},\vec{v})$ is coercive. 
\end{lemma}
\begin{proof}
Since $\vec{Z}(\kdirector)$ is USPD by assumption,
\begin{align*}
\eta K_3 \Ltwoinnerndim{\curl \vec{v}}{\curl \vec{v}}{\Omega}{3} \leq  K_3\Ltwoinnerndim{\vec{Z}(\kdirector) \curl \vec{v}}{\curl \vec{v}}{\Omega}{3},\label{USPDinnerineq}
\end{align*}
where $\eta$ is the relevant lower bound from Lemma \ref{USPDlemma}. Defining $K' = \min (K_1, \eta K_3)>0$ and $\alpha_1 = K'/(C+1)$, where $C = C_4C_3^2$ is the constant defined in \eqref{PFineqforDCuse}, then,
\begin{equation*}
\alpha_1 \Hdcnorm{\vec{v}}{\Omega}^2 \leq K_1 \Ltwoinner{\diverg \vec{v}}{\diverg \vec{v}}{\Omega} + \eta K_3 \Ltwoinnerndim{\curl \vec{v}}{\curl \vec{v}}{\Omega}{3}.
\end{equation*}
Thus, using the assumption that $\klambda$ is pointwise non-negative,
\begin{equation}
\alpha_1 \Hdcnorm{\vec{v}}{\Omega}^2 \leq K_1 \Ltwoinner{\diverg \vec{v}}{\diverg \vec{v}}{\Omega} + K_3 \Ltwoinnerndim{\vec{Z}(\kdirector) \curl \vec{v}}{\curl \vec{v}}{\Omega}{3}+\int_{\Omega} \lambda_k \ltwoinner{\vec{v}}{\vec{v}} \diff{V}. \label{partialcoercivitysmalldata}
\end{equation}
It should be noted that the constant $\eta$ may depend on $\kappa$. Thus, the following three cases are considered.
\begin{caseof}
\case{$\kappa=1+\epsilon_1$, for $\epsilon_1 >0$.}{If this case holds, then $\eta=1$. Hence, $\alpha_1$, defined for \eqref{partialcoercivitysmalldata}, is independent of $\kappa$. Since $K_2-K_3 = K_3(\kappa-1)$, the discrete bilinear form of \eqref{auvform} becomes
\begin{align}
a(\vec{v}, \vec{v}) =& K_1 \Ltwoinner{\diverg \vec{v}}{\diverg \vec{v}}{\Omega} + K_3 \Ltwoinnerndim{\vec{Z}(\kdirector) \curl \vec{v}}{\curl \vec{v}}{\Omega}{3} \nonumber \\
&+ \epsilon_1 K_3 \Big(2 \Ltwoinner{\vec{v} \cdot \curl \vec{v}}{\kdirector \cdot \curl \kdirector}{\Omega} + 2\Ltwoinner{\kdirector \cdot \curl \vec{v}}{\vec{v} \cdot \curl \kdirector}{\Omega} \nonumber \\
&+\Ltwoinner{\vec{v} \cdot \curl \kdirector}{\vec{v} \cdot \curl \kdirector}{\Omega} \Big) + \int_{\Omega} \klambda \ltwoinner{\vec{v}}{\vec{v}} \diff{V} \label{avvsmalldata1}.
\end{align}
Observe that from \eqref{partialcoercivitysmalldata},
\begin{align}
\alpha_1 \Hdcnorm{\vec{v}}{\Omega}^2 \leq& K_1 \Ltwoinner{\diverg \vec{v}}{\diverg \vec{v}}{\Omega} + K_3 \Ltwoinnerndim{\vec{Z}(\kdirector) \curl \vec{v}}{\curl \vec{v}}{\Omega}{3}+\int_{\Omega} \lambda_k \ltwoinner{\vec{v}}{\vec{v}} \diff{V} \nonumber \\
&+\epsilon_1 K_3 \Ltwoinner{\vec{v} \cdot \curl \kdirector}{\vec{v} \cdot \curl \kdirector}{\Omega}. \label{partialcoercivitysmalldatacase1}
\end{align}
Consider the magnitude of the terms in \eqref{avvsmalldata1} not bounded from below in \eqref{partialcoercivitysmalldatacase1}, denoted as $\mathcal{G}(\vec{v}, \vec{v})$,
\begin{align*}
\vert \mathcal{G}(\vec{v}, \vec{v}) \vert &= \vert 2\epsilon_1 K_3 \big(\Ltwoinner{\vec{v} \cdot \curl \vec{v}}{\kdirector \cdot \curl \kdirector}{\Omega}+\Ltwoinner{\kdirector \cdot \curl \vec{v}}{\vec{v} \cdot \curl \kdirector}{\Omega} \big) \vert  \nonumber \\
&\leq 2\epsilon_1 K_3 \big( \vert \Ltwoinner{\vec{v} \cdot \curl \vec{v}}{\kdirector \cdot \curl \kdirector}{\Omega}\vert + \Ltwonorm{\kdirector \cdot \curl \vec{v}}{\Omega} \Ltwonorm{\vec{v} \cdot \curl \kdirector}{\Omega} \big).
\end{align*}
Using bounds derived in the proof of Lemma \ref{bilinearformcontinuity},
\begin{align*}
\vert \mathcal{G}(\vec{v}, \vec{v}) \vert \leq& 4\epsilon_1 K_3\sqrt{\beta C_{\text{sup}}} \Hdcnorm{\vec{v}}{\Omega}^2.
\end{align*}
Denoting $\alpha_3 =4 K_3\sqrt{\beta C_{\text{sup}}}$, then
\begin{equation*}
\vert \mathcal{G}(\vec{v}, \vec{v}) \vert \leq \epsilon_1 \alpha_3 \Hdcnorm{\vec{v}}{\Omega}^2.
\end{equation*}
Utilizing \eqref{partialcoercivitysmalldatacase1},
\begin{equation*}
a(\vec{v}, \vec{v}) \geq \alpha_1 \Hdcnorm{\vec{v}}{\Omega}^2 - \epsilon_1 \alpha_3 \Hdcnorm{\vec{v}}{\Omega}^2 = (\alpha_1-\epsilon_1 \alpha_3) \Hdcnorm{\vec{v}}{\Omega}^2.
\end{equation*}
It is, thus, sufficient to have $\epsilon_1 < \alpha_1/\alpha_3$, guaranteeing that $(\alpha_1-\epsilon_1 \alpha_3)>0$.}
\case{$\kappa=1-\epsilon_2>0$, for $\epsilon_2>0$, and $K_1<K_3$.}
{Since $\kappa<1$, $\eta = 1+(\kappa-1)\beta=(1-\epsilon_2 \beta)$. For $K_1 < K_3$, there exists an $\epsilon_2$ small enough, such that $K_1 < (1-\epsilon_2 \beta)K_3$. This implies that, for small enough $\epsilon_2$,
\begin{equation*}
\alpha_1 = \frac{\min(K_1, (1-\epsilon_2 \beta)K_3)}{(C+1)} = \frac{K_1}{(C+1)}.
\end{equation*}
Therefore, $\alpha_1$ is again independent of $\kappa$. Since $K_2-K_3 = K_3 (\kappa-1)$, the discrete bilinear form of \eqref{auvform} becomes
\begin{align}
a(\vec{v}, \vec{v}) =& K_1 \Ltwoinner{\diverg \vec{v}}{\diverg \vec{v}}{\Omega} + K_3 \Ltwoinnerndim{\vec{Z}(\kdirector) \curl \vec{v}}{\curl \vec{v}}{\Omega}{3} \nonumber \\
&- \epsilon_2 K_3 \Big(2 \Ltwoinner{\vec{v} \cdot \curl \vec{v}}{\kdirector \cdot \curl \kdirector}{\Omega}+2\Ltwoinner{\kdirector \cdot \curl \vec{v}}{\vec{v} \cdot \curl \kdirector}{\Omega}\nonumber \\
&+\Ltwoinner{\vec{v} \cdot \curl \kdirector}{\vec{v} \cdot \curl \kdirector}{\Omega} \Big) + \int_{\Omega} \klambda \ltwoinner{\vec{v}}{\vec{v}} \diff{V} \label{avvsmalldatacase2}.
\end{align}
The terms of \eqref{avvsmalldatacase2}, not already bounded from below in \eqref{partialcoercivitysmalldata}, are bounded as
\begin{align*}
\vert \mathcal{G}(\vec{v}, \vec{v}) \vert &= \vert \epsilon_2 K_3 \big(2 \Ltwoinner{\vec{v} \cdot \curl \vec{v}}{\kdirector \cdot \curl \kdirector}{\Omega} \nonumber \\
&\qquad +2\Ltwoinner{\kdirector \cdot \curl \vec{v}}{\vec{v} \cdot \curl \kdirector}{\Omega}+\Ltwoinner{\vec{v} \cdot \curl \kdirector}{\vec{v} \cdot \curl \kdirector}{\Omega} \big) \vert& \nonumber \\
&\leq \epsilon_2 K_3 \big( 2 \vert \Ltwoinner{\vec{v} \cdot \curl \vec{v}}{\kdirector \cdot \curl \kdirector}{\Omega}\vert \nonumber \\
&\qquad + 2 \Ltwonorm{\kdirector \cdot \curl \vec{v}}{\Omega} \Ltwonorm{\vec{v} \cdot \curl \kdirector}{\Omega} + \Ltwonorm{\vec{v} \cdot \curl \kdirector}{\Omega} \Ltwonorm{\vec{v} \cdot \curl \kdirector}{\Omega} \big).
\end{align*}
Again using the bounds derived in the proof of Lemma \ref{bilinearformcontinuity},
\begin{align*}
\vert \mathcal{G}(\vec{v}, \vec{v}) \vert \leq \epsilon_2 K_3 \big(4 \sqrt{\beta C_{\text{sup}}} + C_{\text{sup}} \big)  \Hdcnorm{\vec{v}}{\Omega}^2.
\end{align*}
Denoting $\alpha_4 =K_3 \big(4 \sqrt{\beta C_{\text{sup}}} + C_{\text{sup}} \big)$, then,
\begin{equation*}
\vert \mathcal{G}(\vec{v}, \vec{v}) \vert \leq \epsilon_2 \alpha_4 \Hdcnorm{\vec{v}}{\Omega}^2.
\end{equation*}
Using \eqref{partialcoercivitysmalldata} implies,
\begin{equation*}
a(\vec{v}, \vec{v}) \geq \alpha_1 \Hdcnorm{\vec{v}}{\Omega}^2 - \epsilon_2 \alpha_4 \Hdcnorm{\vec{v}}{\Omega}^2 = (\alpha_1-\epsilon_2 \alpha_4) \Hdcnorm{\vec{v}}{\Omega}^2.
\end{equation*}
Thus, possibly requiring $\epsilon_2$ to be even smaller, $\epsilon_2 < \alpha_1/\alpha_4$, so that $(\alpha_1-\epsilon_2 \alpha_4)>0$.
 
In the case that $\kappa<1$, the additional restriction that $\beta < \frac{1}{1-\kappa}$ for $\vec{Z}$ to be USPD is necessary, which implies that $\epsilon_2\beta < 1$ is required. Therefore, for any fixed choice of $\beta$, $\epsilon_2$ must also be taken small enough to satisfy this condition. Hence,
\begin{equation*}
\epsilon_2 < \min \left (\frac{\alpha_1}{\alpha_4}, \frac{K_3 - K_1}{\beta K_3}, \frac{1}{\beta} \right).
\end{equation*}
}
\case{$\kappa=1-\epsilon_2>0$, for $\epsilon_2>0$, and $K_3 \leq K_1$.}{Here, again, $\eta = (1-\epsilon_2 \beta)$. For this case, it is clear that $(1-\epsilon_2 \beta)K_3 < K_1$. Thus,
\begin{equation*}
\alpha_1 = \frac{(1-\epsilon_2 \beta) K_3}{(C+1)}.
\end{equation*}
Using the same $\alpha_4$ as in the previous case and similar arguments,
\begin{equation*}
a(\vec{v}, \vec{v}) \geq \alpha_1 \Hdcnorm{\vec{v}}{\Omega}^2 - \epsilon_2 \alpha_4 \Hdcnorm{\vec{v}}{\Omega}^2 = (\alpha_1-\epsilon_2 \alpha_4) \Hdcnorm{\vec{v}}{\Omega}^2.
\end{equation*}
Hence, in order for $(\alpha_1-\epsilon_2 \alpha_4)>0$ to hold, it is necessary that
\begin{equation*}
\epsilon_2 < \frac{K_3}{K_3 \beta + \alpha_4 (C+1)}.
\end{equation*}
Finally, $\epsilon_2$ must still be chosen sufficiently small with respect to $\beta$ such that $\epsilon_2 \beta < 1$, as in Case 2. Therefore,
\begin{equation*}
\epsilon_2 < \min \left (\frac{K_3}{K_3 \beta + \alpha_4 (C+1)}, \frac{1}{\beta} \right).
\end{equation*}
}
\end{caseof}

Thus, if $\epsilon_1$, $\epsilon_2>0$ satisfy the applicable conditions in the cases above, then at each Newton iteration, $a(\vec{u}, \vec{v})$ is coercive for $\kappa \in (1 - \epsilon_2, 1+ \epsilon_1)$.
\end{proof}

\subsection{Discrete Weak Coercivity} \label{DiscreteInfSup}

For this section, we consider the weak coercivity of $b(\cdot, \cdot)$, under Assumption \ref{secass}, with the restriction that $\Omega$ is a polyhedral domain. That is, we show that there exists a $\zeta > 0$ such that
\begin{equation}
\zeta \Ltwonorm{\gamma}{\Omega} \leq \sup_{\vec{v} \in V_h} \frac{\vert b(\vec{v}, \gamma) \vert}{\Hdcnorm{\vec{v}}{\Omega}}, \qquad \forall \gamma \in \Pi_h. \label{discreteLBBcondition}
\end{equation}
Before proving the weak coercivity result for $V_h$ and $\Pi_h$, we prove two critical lemmas. Let $N=2,3$ denote the dimension of $\Omega$.
\begin{lemma} \label{maxbubblelemma}
For the bubble functions, $b_T$, satisfying \eqref{bubbleprop1} and \eqref{bubbleprop2} on a rectangle $T$,  $\displaystyle{\sup_{\vec{x} \in T} b_T =  C_d / \vert T \vert}$, where $C_d=(\frac{3}{2})^N$.
\end{lemma}
\begin{proof}
For $N=2$, without loss of generality, assume that $T$ is a rectangle at the origin given by $[0,a] \times [0, b]$. Let $\bar{b}_T = xy(a-x)(b-y)$ on $T$ and zero elsewhere. Note that $\bar{b}_T$ is the bubble function on $T$ that has not been normalized such that \eqref{bubbleprop1} holds. Integrating over $T$ yields
\begin{align}
\int_T \bar{b}_T \diff{V} =\frac{\vert T \vert^3}{36} \label{bubbleintegral}.
\end{align}
Computing the maximum value of $\bar{b}_T$ shows that $\displaystyle{\sup_{\vec{x} \in T} \bar{b}_T = \frac{\vert T \vert^2}{16}}$. Normalizing $\bar{b}_T$, using \eqref{bubbleintegral}, to define $b_T$ implies that 
\begin{equation*}
\sup_{\vec{x} \in T} b_T = \frac{\vert T \vert^2 / 16}{\vert T \vert^3 / 36} = \frac{9}{4 \vert T \vert}.
\end{equation*}
The case for $N=3$ is derived analogously for $T$, the rectangular box $[0,a]\times[0,b]\times[0,c]$, and $\bar{b}_T = xyz(a-x)(b-y)(c-z)$. The corresponding $b_T$ satisfies
\begin{equation*}
\sup_{\vec{x} \in T} b_T = \frac{\vert T \vert^2 / 64}{\vert T \vert^3 / 216} = \frac{27}{8 \vert T \vert}.
\end{equation*}
\end{proof}

Following the notation in \cite{Brenner1}, consider two finite elements $(T, \mathcal{P}, \mathcal{N})$ and $(\hat{T}, \hat{\mathcal{P}}, \hat{\mathcal{N}})$, where $T$ and $\hat{T}$ are element domains, $\mathcal{P}$ and $\hat{\mathcal{P}}$ are the respective sets of basis functions, and $\mathcal{N}$ and $\hat{\mathcal{N}}$ are the associated dual bases. We say that $(\hat{T}, \hat{\mathcal{P}}, \hat{\mathcal{N}})$ is affine equivalent to $(T, \mathcal{P}, \mathcal{N})$ if there exists an affine mapping, $G: T \to \hat{T}$, such that for $\vec{x} \in T$
\begin{equation*}
G\vec{x} = \vec{x}_0 + M\vec{x},
\end{equation*}
with non-singular matrix $M$, satisfying
\begin{itemize}
\item $G(T) = \hat{T}$
\item $G^* \hat{\mathcal{P}} = \mathcal{P}$ and
\item $G_{*} \mathcal{N} = \hat{\mathcal{N}}$.
\end{itemize}
Here, the pullback $G^*$ is defined by $G^*(\hat{f}) := \hat{f} \circ G$, and the push-forward $G_{*}$ is defined by $(G_{*} N)(\hat{f}) := N(G^*(\hat{f}))$.

\begin{lemma} \label{affinebubblelemma}
Consider a rectangular reference element $(T, \mathcal{P}, \mathcal{N})$, where $\mathcal{P}$ is the basis of shape functions for $T$ associated with $V_h \times \Pi_h$, defined above. If, for all $\hat{T} \in \mathcal{T}_h$, $(\hat{T}, \hat{\mathcal{P}}, \hat{\mathcal{N}})$ is affine equivalent to $(T, \mathcal{P}, \mathcal{N})$, then $\displaystyle{\sup_{\hat{\vec{x}} \in \hat{T}} b_{\hat{T}} = C_d/ \vert \hat{T} \vert}$, where $b_{\hat{T}}$ is the normalized bubble function satisfying \eqref{bubbleprop1} and \eqref{bubbleprop2} on $\hat{T}$.
\end{lemma}
\begin{proof}
Note that the non-normalized bubble function on $\hat{T}$, $\bar{b}_{\hat{T}}$, is given by
\begin{equation*}
\bar{b}_{\hat{T}} = b_T \circ G^{-1},
\end{equation*}
where $b_T$ is the normalized bubble function on $T$. Therefore, the maximum value for $\bar{b}_{\hat{T}}$ corresponds to the maximum value for $b_T$, which, as shown in Lemma \ref{maxbubblelemma}, is $C_d/\vert T \vert$. Observe that
\begin{align*}
\int_{\hat{T}} \bar{b}_{\hat{T}} \diff{V} &= \int_T b_T \vert \det M \vert \diff{V} \nonumber \\
&= \vert \det M \vert,
\end{align*}
where $\det M$ denotes the determinant of the matrix $M$. Thus, $b_{\hat{T}}$ is given by dividing $\bar{b}_{\hat{T}}$ by $\vert \det M \vert$. Therefore,
\begin{align*}
\sup_{\hat{\vec{x}} \in \hat{T}} b_{\hat{T}} &= \frac{1}{\vert \det M \vert} \sup_{\vec{x} \in T} b_T \nonumber \\
&=\frac{C_d}{\vert \det M \vert \vert T \vert} \nonumber \\
&= \frac{C_d}{\vert \hat{T} \vert}.
\end{align*}
\end{proof}

In the following, we will make use of the following second set of assumptions when necessary.
\begin{assumption} \label{secass2}
Let $\{ \triangulation \}$ be a family of quadrilateral subdivisions of a polyhedral domain $\Omega$ satisfying \eqref{maxuniformity} and \eqref{quasiuniform}. Moreover, assume that for each $T \in \triangulation$, the element $(T, \mathcal{P}_T, \mathcal{N}_T)$ is affine equivalent to a rectangular reference element for all $h$.
\end{assumption}

Prior to considering the following lemma, recall that $\alpha$ and $\beta$ are the bounds on the length of $\director$ in \eqref{limitsonnlength}, $\rho$ is the quasi-uniform mesh parameter defined in \eqref{quasiuniform}, and $C_d$ is the constant derived in Lemma \ref{maxbubblelemma} depending on $N$, the dimension of $\Omega$.
\begin{lemma} \label{bubblespacelemma}
Under Assumptions \ref{secass} and \ref{secass2}, $V_h$ and $\Pi_h$ constitute a pair satisfying \eqref{discreteLBBcondition} with constant $\zeta = h \left[\frac{2 \alpha \rho^N}{9C_f C_{*} \sqrt{\beta C_d}} \right]$, for $C_f$ and $C_{*}$ defined below.
\end{lemma}
\begin{proof}
Since $V_h \subset Q_l \times Q_l \times Q_l$, by \cite[Theorem 4.5.11]{Brenner1} there exists $C_{*}>0$ depending only on $\rho$ such that
\begin{equation*}
\Honenorm{\vec{v}}{\Omega} \leq C_{*} h^{-1} \Ltwonormndim{\vec{v}}{\Omega}{3}.
\end{equation*}
Furthermore, using the fact that $\Hdcnorm{\vec{v}}{\Omega} \leq C_f \Honenorm{\vec{v}}{\Omega}$,
\begin{equation} \label{Cfreference}
\sup_{\vec{v} \in V_h} \frac{\vert b(\vec{v}, \gamma) \vert}{\Hdcnorm{\vec{v}}{\Omega}} \geq \sup_{\vec{v} \in V_h} \frac{\vert b(\vec{v}, \gamma) \vert}{C_f \Honenorm{\vec{v}}{\Omega}} \geq \sup_{\vec{v} \in V_h} \frac{\vert b(\vec{v}, \gamma) \vert}{C_f C_{*} h^{-1} \Ltwonormndim{\vec{v}}{\Omega}{3}}.
\end{equation}
Therefore, \eqref{discreteLBBcondition} is reduced to finding $\zeta>0$ such that
\begin{align*}
\zeta \Ltwonorm{\gamma}{\Omega} \leq \sup_{\vec{v} \in V_h} \frac{\vert b(\vec{v}, \gamma) \vert}{C_f C_{*} h^{-1} \Ltwonormndim{\vec{v}}{\Omega}{3}}, \qquad \forall \gamma \in \Pi_h.
\end{align*}

Now consider constructing $\vec{v}_0$ on each $T \in \triangulation$ by letting $a_T = \gamma |_T$, where this denotes the restriction of $\gamma$ to the element $T$, and defining
\begin{equation*}
\vec{v}_0 |_T = a_T b_T \kdirector |_T.
\end{equation*}
Observe that, as defined, $\vec{v}_0 \in V_h$. Let $C_m = \max_{T \in \triangulation} \vert T \vert$. Then,
\begin{align}
b(\vec{v}_0, \gamma) = \sum_{T \in \triangulation} \int_T \gamma \ltwoinner{\vec{v}_0}{\kdirector} &\geq \alpha \sum_{T \in \triangulation} \gamma^2 \int_T b_T \diff{V} \nonumber \\
&= \alpha \sum_{T \in \triangulation} \gamma^2 \geq \frac{\alpha}{C_m} \Ltwonorm{\gamma}{\Omega}^2. \label{discreteinfsupnumerator}
\end{align}
It is also the case that
\begin{align*}
\Ltwonormndim{\vec{v}_0}{\Omega}{3}^2 &= \sum_{T \in \triangulation} \int_T a_T^2 b_T^2 \ltwoinner{\kdirector}{\kdirector} \diff{V} \leq \beta \sum_{T \in \triangulation} \gamma^2 \int_T b_T^2 \diff{V}.
\end{align*}
Since the bubble functions are fixed, let 
\begin{align*}
C_b = \max_{T \in \triangulation} \int_T b_T^2 \diff{V}, \qquad C_T = \min_{T \in \triangulation} \vert T \vert.
\end{align*}
Thus,
\begin{align}
\Ltwonormndim{\vec{v}_0}{\Omega}{3}^2 &\leq \beta C_b \sum_{T \in \triangulation} \gamma^2 \leq \frac{\beta C_b}{C_T} \Ltwonorm{\gamma}{\Omega}^2.\label{discreteinfsupdenominator}
\end{align}
Therefore, combining \eqref{discreteinfsupnumerator} and \eqref{discreteinfsupdenominator}, 
\begin{align} 
\sup_{\vec{v} \in V_h} \frac{\int_{\Omega} \gamma \ltwoinner{\vec{v}}{\kdirector} \diff{V}}{\Ltwonorm{\vec{v}}{\Omega}} &\geq  \frac{\int_{\Omega} \gamma \ltwoinner{\vec{v}_0}{\kdirector} \diff{V}}{\Ltwonorm{\vec{v}_0}{\Omega}}\nonumber \\ 
&\geq \frac{\frac{\alpha}{C_m} \Ltwonorm{\gamma}{\Omega}^2}{\sqrt{\frac{\beta C_b}{C_T}} \Ltwonorm{\gamma}{\Omega}} = \frac{\alpha \sqrt{C_T}}{C_m \sqrt{\beta C_b}} \Ltwonorm{\gamma}{\Omega} \label{meshdepinfsupconstant}.
\end{align}
Note that the final constant in \eqref{meshdepinfsupconstant} is mesh dependent. Let $N=2,3$ denote the dimension of $\Omega$. Observe that
\begin{equation*}
C_b \leq \max_{T \in \triangulation} \sup_{\vec{x} \in T} b_T \int_T b_T \diff{V} = \max_{T \in \triangulation} \sup_{\vec{x} \in T} b_T.
\end{equation*}
From Lemma \ref{affinebubblelemma}, for arbitrary $T \in \triangulation$,
\begin{equation*}
\sup_{\vec{x} \in T} b_T = C_d/ \vert T \vert,
\end{equation*}
where $C_d$ depends only on the dimension of $\Omega$. Therefore,
\begin{equation*}
\max_{T \in \triangulation} \sup_{\vec{x} \in T} b_T = \frac{C_d}{C_T}.
\end{equation*}
Hence,
\begin{equation} \label{CBinequal}
\frac{\sqrt{C_T}}{C_m \sqrt{C_b}} \geq \frac{C_T}{C_m \sqrt{C_d}}.
\end{equation}
Define the constants
\begin{align*}
&C_{2,1} = \frac{\pi}{4},&  &C_{2,2} = \pi, & &\text{for } N = 2, \\
&C_{3,1} = \frac{\pi}{6}, & &C_{3,2} = \frac{3\pi}{4},&  &\text{for } N=3.
\end{align*}
Using Properties \eqref{maxuniformity} and \eqref{quasiuniform} with the constants above, it is straightforward to show that
\begin{align*}
C_T &\geq C_{N,1} (\min \{ \diam B_T : T \in \triangulation \})^N \geq C_{N,1} \rho^N (h \diam \Omega)^N, \\
C_m &\leq C_{N,2} (\max \{ \diam T : T \in \triangulation \})^N \leq C_{N,2} (h \diam \Omega)^N.
\end{align*}
Therefore,
\begin{align} \label{CTCMInequal}
\frac{C_T}{C_m} \geq \frac{C_{N,1} \rho^N}{C_{N,2}}.
\end{align}
Utilizing \eqref{CBinequal} and \eqref{CTCMInequal}
\begin{equation*}
\frac{\alpha \sqrt{C_T}}{C_m \sqrt{\beta C_b}} \Ltwonorm{\gamma}{\Omega} \geq \frac{\alpha C_{N,1}\rho^N}{C_{N,2} \sqrt{\beta C_d}} \Ltwonorm{\gamma}{\Omega} \geq \frac{2\alpha \rho^N}{9\sqrt{\beta C_d}} \Ltwonorm{\gamma}{\Omega},
\end{equation*}
where $C_d$ depends only on the dimension of $\Omega$. Hence, \eqref{discreteLBBcondition} is satisfied with constant $\zeta = h \left[\frac{2 \alpha \rho^N}{9C_f C_{*} \sqrt{\beta C_d}} \right]$. Thus, $V_h$ and $\Pi_h$ represent a pair of spaces on which $b(\cdot, \cdot)$ is weakly coercive.
\end{proof}\\
For $\kdirector \in Q_p$, with $V_h \subset Q_m \times Q_m \times Q_m \oplus V_h^b$, as in \eqref{vspace}, and $l = \max (m, p+2)$, the above lemma yields an immediate corollary.
\begin{corollary} \label{stabilitycorollary}
Under Assumptions \ref{secass} and \ref{secass2}, $\kdirector \in Q_p$ implies that $b(\cdot, \cdot)$ is weakly coercive for the pair $Q_l$--$P_0$. The special case that $\kdirector \in P_0$ implies that $b(\cdot, \cdot)$ is weakly coercive on the pair $Q_{\max(m, 2)}$--$P_0$.
\end{corollary}
\begin{proof}
Note that if $\kdirector \in Q_p$, the bubble space defined above satisfies $V_h^b \subset Q_{p+2} \times Q_{p+2} \times Q_{p+2}$, since $b_T \in Q_2$. This implies that $V_h \subset Q_l \times Q_l \times Q_l$. Therefore, since $b(\cdot, \cdot)$ is weakly coercive for the pair $V_h$--$P_0$, weak coercivity must also hold for the pair $Q_l$--$P_0$. If $\kdirector \in P_0$, then $V_h^b \subset Q_2 \times Q_2 \times Q_2$. Hence, $V_h \subset Q_{\max(m, 2)} \times Q_{\max(m, 2)} \times Q_{\max(m, 2)}$. The lemma above is equally valid for $\kdirector \in P_0$. Therefore, $b(\cdot, \cdot)$ is weakly coercive on the pair $Q_{\max(m, 2)}$--$P_0$ for the given $\kdirector$.
\end{proof}

In light of the lemmas discussed above, verification of weak coercivity allows for the formulation and proof of this paper's main theorem.
\begin{theorem} \label{existuniquetheorem}
Under Assumptions \ref{secass} and \ref{secass2}, existence of discrete solutions $(\ddirector_h, \dlambda_h)$ for each Newton linearization are guaranteed for the pair $V_h$--$\Pi_h$. In the case that $\kappa=1$ or that $\kappa$ satisfies the small data conditions of Lemma \ref{coercivitysmalldata}, such solutions are unique.
\end{theorem}
\begin{proof}
Following a mixed formulation approach based on \cite{Brenner1, Braess1, Boffi1}, Lemmas \ref{boundedlinearforms} and \ref{bilinearformcontinuity} guarantee the existence of a solution to the system given by \eqref{generalizedNewtoniterationweakform1} and \eqref{generalizedNewtoniterationweakform2}. In the event that $\kappa =1$ or that $\kappa$ satisfies the small data assumptions, Lemma \ref{coercivityauv} or \ref{coercivitysmalldata} coupled with Lemma \ref{bubblespacelemma} implies that the solution is also unique.
\end{proof}

\subsection{Error Analysis}

In the previous section, the derived weak coercivity constant depends on the mesh parameter $h$. Therefore, as $h$ approaches zero so too does the weak coercivity constant for the pair $V_h$ and $\Pi_h$. However, the convergence of the scheme for the enriched Lagrangian finite-element spaces composing $V_h$ is only slightly compromised. In this section, we derive approximation error bounds for the discrete solution. Throughout this section, it is assumed that Assumptions \ref{secass} and \ref{secass2} apply. Let $(\vec{u}, q)$ represent a solution to the continuum variational system given by \eqref{contgeneralizedNewtoniterationweakform1} and \eqref{contgeneralizedNewtoniterationweakform2} and $(\vec{u}_h, q_h)$ be the unique solution to the discrete system in \eqref{generalizedNewtoniterationweakform1} and \eqref{generalizedNewtoniterationweakform2}. As above, denote the dimension of $\Omega$ by $N=2, 3$.

\begin{lemma}
Let $\Pi_h$ and $V_h$ be defined as in \eqref{pispace} and \eqref{vspace} with $m=2$. Under Assumptions \ref{secass} and \ref{secass2}, for $\vec{u} \in \Hn{3}{\Omega}^3$ and $q \in \Hone{\Omega}$ there exists $C_a>0$ such that
\begin{equation} \label{errorapprox}
\Hdcnorm{\vec{u}-\vec{u}_h}{\Omega} \leq C_a h \big (\Hnnorm{\vec{u}}{3} + \Honenorm{q}{\Omega} \big ).
\end{equation}
\end{lemma}
\begin{proof}
Let $\alpha_0$ denote the coercivity constant from either Lemma \eqref{coercivityauv} or \eqref{coercivitysmalldata}. Furthermore, let $\zeta$ denote the $h$-dependent weak coercivity constant derived in Lemma \eqref{bubblespacelemma}. By Theorem 5.2.2 in \cite{Boffi1},
\begin{equation} \label{boffiinequality}
\Hdcnorm{\vec{u} - \vec{u}_h}{\Omega} \leq \frac{4 C_A C_B}{\alpha_0 \zeta} E_u + \frac{C_B}{\alpha_0} E_{q},
\end{equation}
where $C_A$ and $C_B$ are the continuity constants associated with $a(\cdot, \cdot)$ and $b(\cdot, \cdot)$, respectively, and
\begin{align*}
E_u = \inf_{v_h \in V_h} \Hdcnorm{\vec{u} - \vec{v}_h}{\Omega}, & & E_{q} = \inf_{\gamma_h \in \Pi_h} \Ltwonorm{q - \gamma_h}{\Omega}.
\end{align*}
Note that
\begin{equation*}
\inf_{v_h \in V_h} \Hdcnorm{\vec{u} - \vec{v}_h}{\Omega} \leq C_f \inf_{v_h \in V_h} \Honenorm{\vec{u}-\vec{v}_h}{\Omega},
\end{equation*}
where $C_f$ is the constant used in \eqref{Cfreference}. Let $\mathcal{I}^h f$ denote the global interpolant of $f$ over the appropriate finite-element space. Since $\{\triangulation\}$ is quasi-uniform, it is, in particular, non-degenerate. Therefore, applying \cite[Theorem 4.4.24]{Brenner1} to the discrete space $V_h$, there exists a $C_5>0$, such that
\begin{align*}
\left ( \sum_{T \in \triangulation} \Vert \vec{v} - \mathcal{I}^h \vec{v} \Vert^2_{H^1(T)} \right)^{1/2} = \Honenorm{\vec{v} - \mathcal{I}^h \vec{v}}{\Omega} \leq C_5 h^2 \Hnnorm{\vec{v}}{3},& &\forall \vec{v} \in \Hn{3}{\Omega}.
\end{align*}
This implies that if $\vec{u} \in \Hn{3}{\Omega}^3$, then
\begin{equation} \label{Vhbound}
\inf_{\vec{v}_h \in V_h} \Hdcnorm{\vec{u} - \vec{v}_h}{\Omega} \leq C_f C_5 h^2 \Hnnorm{\vec{u}}{3}.
\end{equation}

For $\Pi_h$, Theorem 3.1.6 in \cite{Ciarlet1} implies that there exists a $C_6>0$ such that
\begin{align*}
\Ltwonorm{\gamma-\mathcal{I}^h \gamma}{\Omega} \leq C_6 h \Honenorm{\gamma}{\Omega}, & &\forall \gamma \in \Hone{\Omega}.
\end{align*}
Hence, if $q \in \Hone{\Omega}$,
\begin{equation} \label{Pihbound}
\inf_{\gamma_h \in \Pi_h} \Ltwonorm{q-\gamma_h}{\Omega} \leq C_6 h \Honenorm{q}{\Omega}.
\end{equation}
Combining \eqref{Vhbound} and \eqref{Pihbound} with \eqref{boffiinequality} yields the error estimate
\begin{align*}
\Hdcnorm{\vec{u}-\vec{u}_h}{\Omega} &\leq \frac{4 C_A C_B}{\alpha_0 \zeta} C_f C_5 h^2 \Hnnorm{\vec{u}}{3} + \frac{C_B}{\alpha_0} C_6 h \Honenorm{q} \nonumber \\
&= \frac{18C_A C_B C_f^2 C_{*} \sqrt{\beta C_d} C_5}{\alpha \rho^N \alpha_0} h \Hnnorm{\vec{u}}{3} + \frac{C_B C_6}{\alpha_0} h\Honenorm{q}{\Omega}.
\end{align*}
Taking $C_a = \max \left ( \frac{18C_A C_B C_f^2 C_{*} \sqrt{\beta C_d} C_5}{\alpha \rho^N \alpha_0}, \frac{C_B C_6}{\alpha_0} \right)$, \eqref{errorapprox} is obtained.
\end{proof}

Thus, the approximation is convergent for $V_h$--$\Pi_h$ but with an order of sub-optimality, due to the weak coercivity constant's dependence on the mesh parameter. However, use of a discrete $\Hn{-1}{\Omega}$ norm for the space $\Pi_h$ is currently being considered as a means of eliminating this mesh dependence. 

\subsection{Practical Choice of Finite Elements}

The bubble enrichment discussed above is non-standard in its incorporation of $\kdirector$ in the construction of the bubbles. Therefore, during numerical implementation, it was desirable to find an experimentally stable, standard, finite-element pair closely related to the spaces discussed above. It was observed that $Q_1$--$Q_1$ finite-element discretizations resulted in singular matrices. This implies that $Q_1$--$Q_1$ is not a pair for which $b(\cdot, \cdot)$ is weakly coercive. Such a phenomenon is not unique. For example, instabilities arise for equal order elements in Galerkin approaches to both the Stokes' equations \cite{Bochev1} and the Navier-Stokes' equations \cite{Franca1}. 

On the other hand, in the numerical experiments to be discussed below, mixed finite-element approaches, such as $Q_2$--$P_0$ discretizations, experimentally appear to admit weak coercivity without the need for rising order finite-element implementations or bubble enrichments. Corollary \ref{stabilitycorollary} implies that for a piecewise constant initial iterate, the update element space $Q_2$--$P_0$ implies weak coercivity for the first Newton iteration. With this assurance, coupled with the empirical weak coercivity evidence for $Q_2$--$P_0$, we employ $Q_2$--$P_0$ spaces for all iterations in the experiments below. In the event that singular matrices occur for the $Q_2$--$P_0$ discretization of a particular problem, the bubble enriched finite-element pair $V_h$--$\Pi_h$, defined in \eqref{pispace} and \eqref{vspace}, may be implemented and is particularly attractive because the rising order of the bubble functions, $b_T \kdirector \vert_T$, on each element does not increase the number of unknowns at each Newton iteration.

\section{Numerical Methodology} \label{nummethodology}

The algorithm to perform the minimization discussed in previous sections has three stages; see Algorithm \ref{algo}. The outermost phase is nested iteration (NI) \cite{McCormick1, Starke1}, which begins on a specified coarsest grid level. Newton iterations are performed on each grid, updating the current approximation after each step. The stopping criterion for the Newton iterations at each level is based on a specified tolerance for the current approximation's conformance to the first-order optimality conditions in the standard Euclidean $l_2$ norm. In the numerical experiments to follow, this tolerance was always $10^{-3}$. The resulting approximation is then interpolated to a finer grid. The current implementation performs uniform grid refinement after each set of Newton iterations.

The Newton iteration systems are constructed by applying finite-element discretizations on each grid. The resulting, relatively sparse, matrix has the anticipated saddle-point block structure
\begin{equation*}
\left [ \begin{array}{c c}
\vec{A} & \vec{B} \\
\vec{B}^{T} & \vec{0} \end{array} \right ] .
\end{equation*}
The matrix is inverted using LU decomposition in order to solve for the discrete updates $\ddirector_h$ and $\dlambda_h$. Finally, an incomplete Newton correction is performed. That is, the new iterates are given by
\begin{equation} \label{corrections}
 \left [ \begin{array}{c}
 \director_{k+1} \\
 \lambda_{k+1} 
 \end{array} \right ]
= \left [ \begin{array}{c} 
\kdirector \\
\klambda \\
\end{array} \right ] + \omega
\left [ \begin{array}{c}
\ddirector_h \\
\dlambda_h
\end{array} \right ],
 \end{equation}
where $\omega \leq 1$. This is to ensure relatively strict adherence to the constraint manifold, which is necessary for the well-posedness discussed above. For this algorithm, $\omega$ is chosen to begin at $0.2$ on the coarsest grid and increases by $0.2$, to a maximum of $1$, after each grid refinement, so that as the approximation converges, larger Newton steps are taken. For complicated boundary conditions, such damped Newton steps are important in preventing method divergence. The grid management and discretizations are implemented using the deal.II finite-element library, which is an aggressively optimized and parallelized open-source library widely used in scientific computing \cite{BangerthHartmannKanschat2007, DealIIReference}. In practice, as discussed above, $Q_2$--$P_0$ discretizations were observed to experimentally admit weak coercivity. Therefore, $Q_2$--$P_0$ elements were used to approximate $\ddirector_h$ and $\dlambda_h$ on each grid for the numerical tests.

\vspace{.3cm}
\begin{algorithm}[H] \label{algo}
\SetAlgoLined
~\\
0. Initialize $(\director_0, \lambda_0)$ on coarse grid.
~\\
\While{Refinement limit not reached}
{
	\While{First-order optimality conformance threshold not satisfied}
	{
		1. Set up discrete linear system \eqref{newtonhessian} on current grid, $H$. ~\\
		2. Solve for $\ddirector_{H}$ and $\dlambda_{H}$. ~\\
		3. Compute $\director_{k+1}$ and $\lambda_{k+1}$ as in \eqref{corrections}. ~\\
	}
	4. Uniformly refine the grid. ~\\
	5. Interpolate $\director_{H} \to \director_{h}$ and $\lambda_{H} \to \lambda_h$.
}
\caption{Newton's method minimization algorithm with NI}
\end{algorithm}
\vspace{.3cm}

\subsection{Free Elastic Numerical Results}\label{FreeElasticResults}

The general test problem in this section considers a classical domain with two parallel substrates placed at distance $d=1$ apart. The substrates run parallel to the $xz$-plane and perpendicular to the $y$-axis. It is assumed that this domain represents a uniform slab in the $xy$-plane. That is, $\director$ may have a non-zero $z$ component but $\pd{\director}{z} = \vec{0}$. Hence, we consider the 2-D domain $\Omega = \{ (x,y) \text{ } \vert \text{ } 0 \leq x,y \leq 1 \}$. The problem assumes periodic boundary conditions at the edges $x=0$ and $x=1$. Dirichlet boundary conditions are enforced on the $y$-boundaries. As discussed above, the simplification outlined in \eqref{stronganchoringdivthm} is relevant for this domain and boundary conditions.

\begin{figure}[h!]
\centering
\begin{subfigure}{.49 \textwidth}
\centering
  \includegraphics[scale=.30]{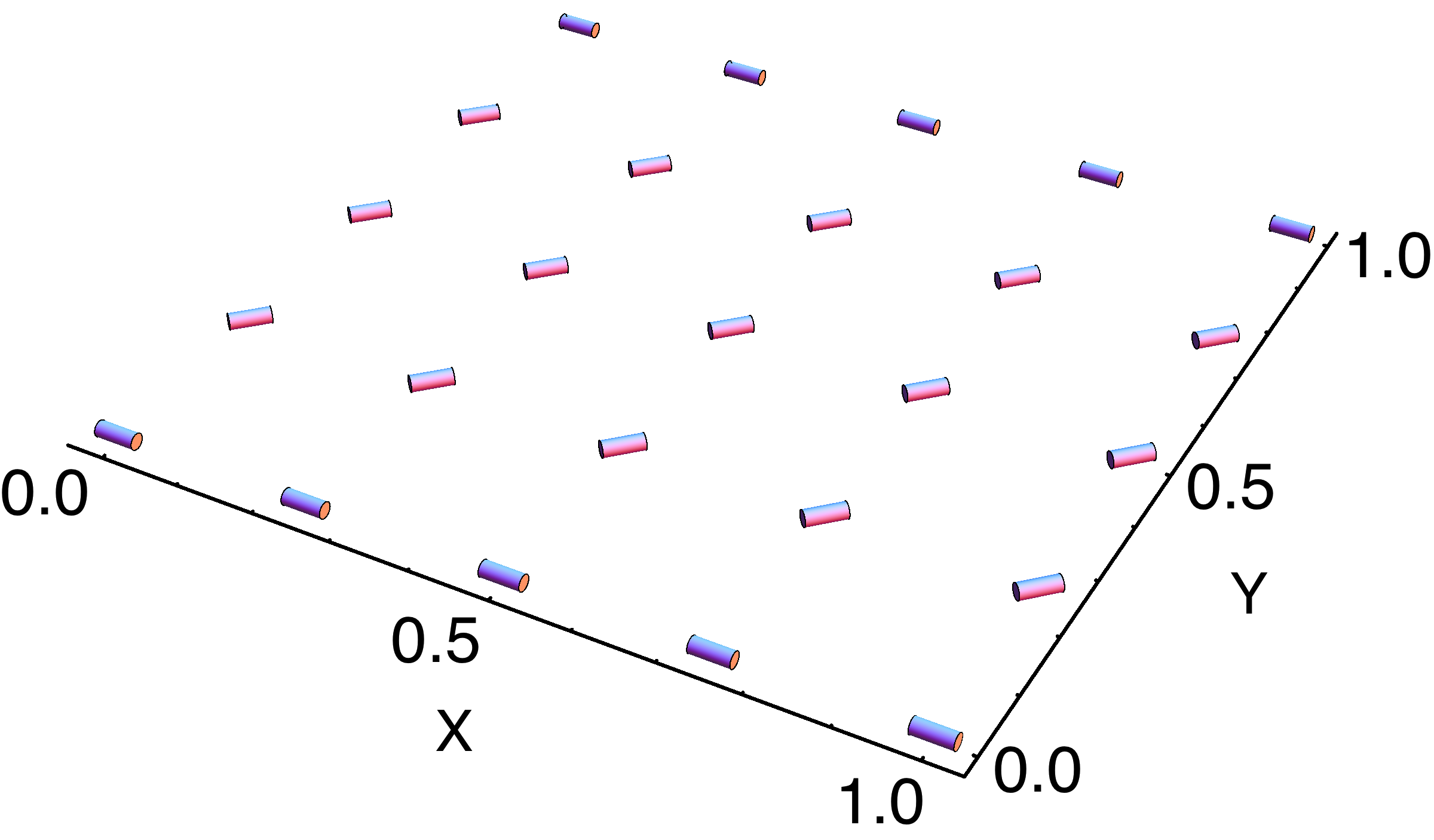}
\end{subfigure}
\begin{subfigure}{.49 \textwidth}
\raggedright
  \includegraphics[scale=.30]{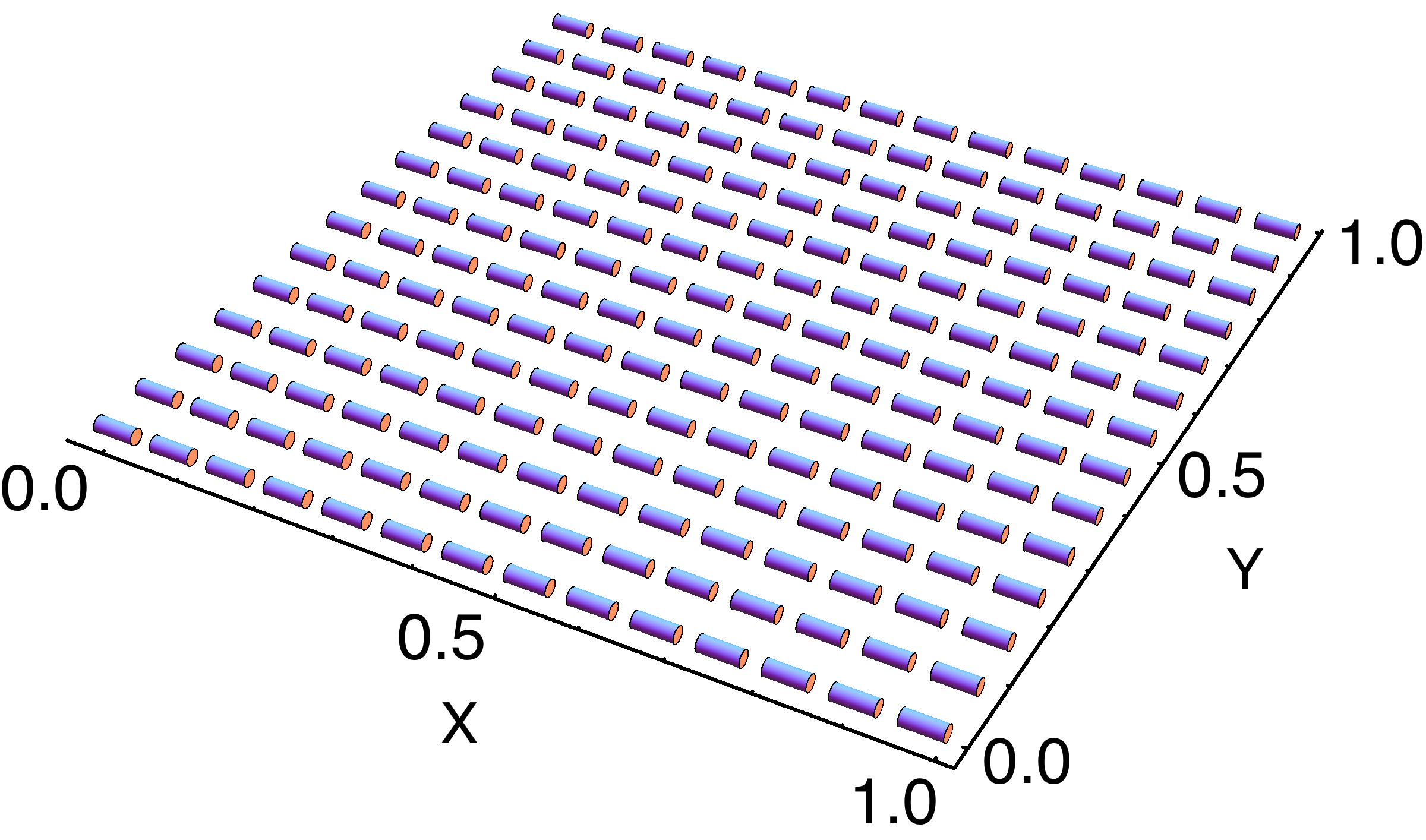}
\end{subfigure}
\caption{\small{Initial guess (left) on $4 \times 4$ mesh with initial free energy of $5.467$ and resolved solution (right) on $128 \times 128$ mesh (restricted for visualization) with final free energy of 0 for a uniformly aligned boundary.}}
\label{FEsimpleBC}
\end{figure}

The first numerical experiment is run on one of the simplest configurations of this type. Along each of the substrates the liquid crystal rods are uniformly aligned parallel to the $x$-axis. The relevant Frank constants are $K_1=K_2=K_3=1$. The problem is solved on a $4 \times 4$ coarse grid with five successive uniform refinements resulting in a $128 \times 128$ fine grid. The initial guess and computed, converged solution are displayed in Figure \ref{FEsimpleBC}.

The final minimized functional energy is $\mathcal{F}_1 = 0$, compared to the initial guess energy of $5.467$. In Table \ref{gridprogFEsimpleBC}, the number of Newton iterations per grid is detailed as well as the conformance of the solution to the first-order optimality conditions after the first and final Newton steps, respectively, on each grid. Assuming the presence of solvers that scale linearly with the number of non-zeros in the matrix, the work required in these iterations is roughly $1.34$ times that of assembling and solving a single linearization step on the finest grid. In contrast, without nested iteration, the algorithm requires $21$ damped Newton steps on the $128 \times 128$ finest grid alone, to satisfy the tolerance limit. The application of damped Newton steps becomes even more important when beginning on finer grids with a rough initial guess, as divergence can be more prevalent. Table \ref{gridprogFEsimpleBC} also reveals the performance of the algorithm with respect to the pointwise constraint, presenting the increasingly tighter minimum and maximum director deviations from unit length at the quadrature nodes. The computed equilibrium solution behaves as expected with the rods uniformly aligning parallel to the $x$-axis.

\begin{table}[h!]
\centering
{\small
\begin{tabular}{|c|c|c|c|c|c|}
\hline
Grid Dim. &Newton Iter.&Init. Res.&Final Res.& Deviation in $\ltwonorm{\director}^2$ & Final Energy\\
\hline
$4 \times 4$ & 18 & 4.35e-00 & 4.39e-04 &6.17e-06, 5.54e-05 & 4.941e-08\\
\hline
$8 \times 8$ & 1 & 2.44e-04 & 9.74e-05 & 1.25e-06, 2.26e-05 & 7.905e-09\\
\hline
$16 \times 16$ & 1 & 5.48e-05 & 1.10e-05 & 1.26e-07, 4.55e-06 & 3.162e-10\\
\hline
$32 \times 32$ & 1 & 6.42e-06 & 1.35e-11 & 4.20e-14, 4.30e-11 & 7.932e-21\\
\hline
$64 \times 64$ & 1 & 6.77e-12 & 6.37e-14 & -4.00e-16, 0 & 0\\
\hline
$128 \times 128$ & 1 & 1.30e-13 & 1.14e-13 & -4.00e-16, 0 & 0\\
\hline
\end{tabular}
}
\caption{\small{Grid and solution progression for uniform free elastic boundary conditions with initial and final residuals for the first-order optimality conditions, minimum and maximum director deviations from unit length at the quadrature nodes, and final functional energy on each grid.}}
\label{gridprogFEsimpleBC}
\end{table}

The second test, run for the free elastic slab problem, incorporates twist boundary conditions and unequal Frank constants. On the lower slab, along $y=0$, the nematic rods are aligned parallel to the $x$-axis. For the upper slab, the rods are uniformly aligned along the $z$-axis. The relevant constants for this run are $K_1 = 1$, $K_2 = 1.2$, and $K_3 = 1$. This implies that $\kappa = K_2/K_3 > 1$. The solves are again performed on a $4 \times 4$ coarse grid, uniformly ascending to a $128 \times 128$ fine grid. The expected configuration for such boundary conditions is a twisted equilibrium solution along the $y$-axis. Indeed, the numerically resolved solution in Figure \ref{FETwistedBC}, displayed alongside the initial guess, demonstrates such a twist. The final minimized functional energy is $\mathcal{F}_1 = 1.480$, compared to the initial guess energy of $12.534$. Table \ref{gridprogFETwistedBC} enumerates the algorithm run attributes. 

\begin{figure}[h!]
\centering
\begin{subfigure}{.49 \textwidth}
\centering
  \includegraphics[scale=.30]{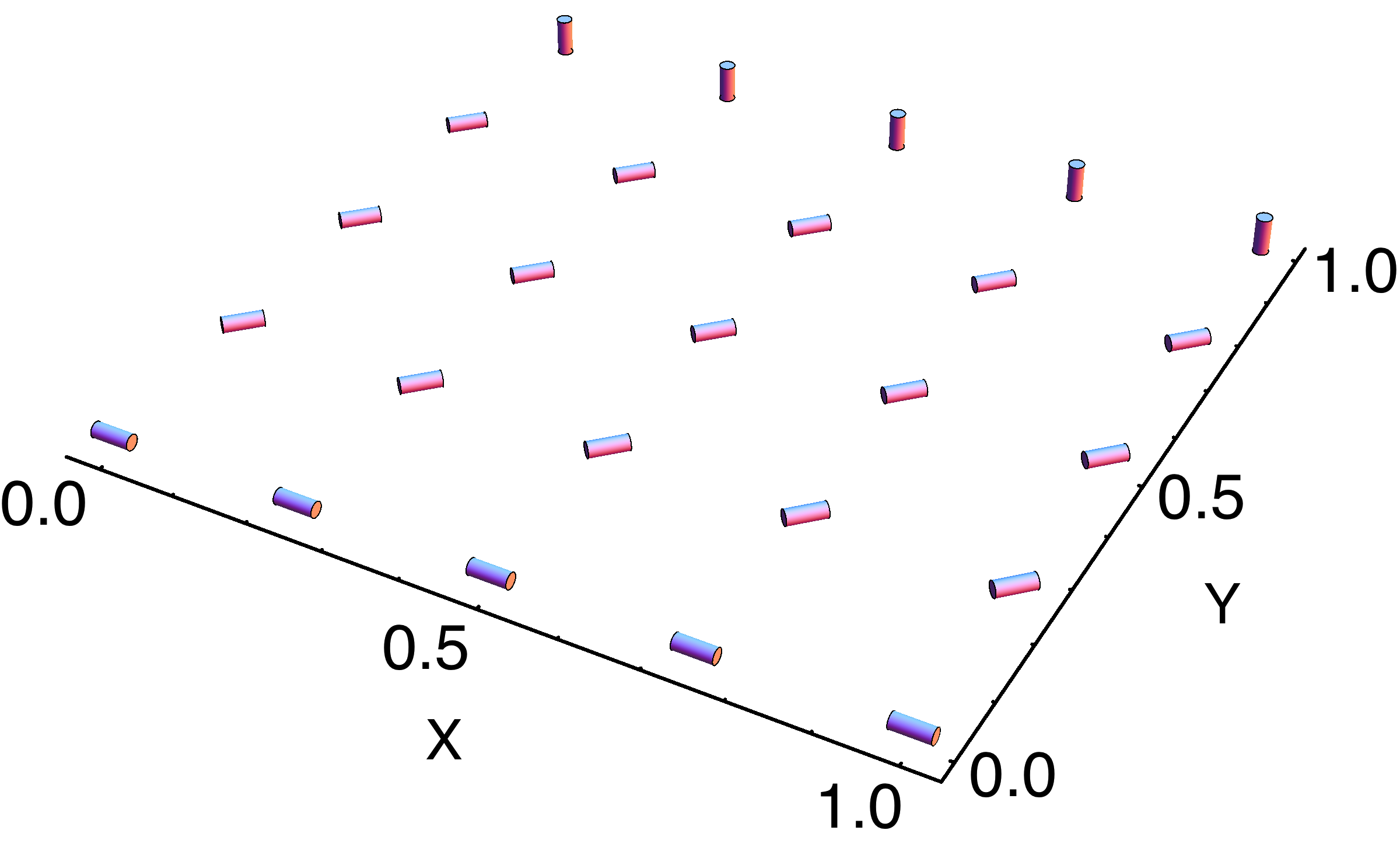}
\end{subfigure}
\begin{subfigure}{.49 \textwidth}
\raggedright
  \includegraphics[scale=.30]{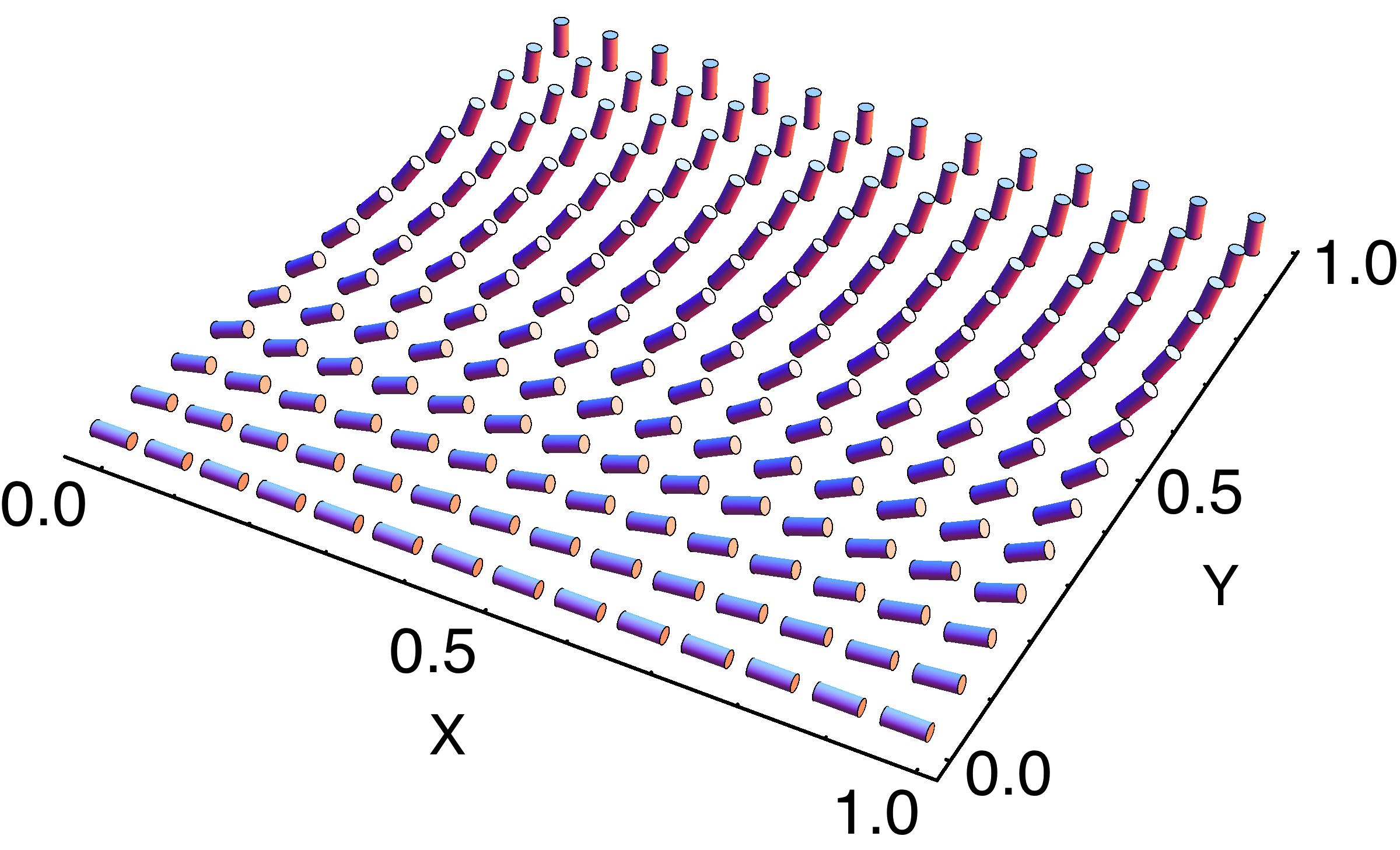}
\end{subfigure}
\caption{\small{Initial guess (left) on $4 \times 4$ mesh with initial free energy of $12.534$ and resolved solution (right) on $128 \times 128$ (mesh restricted for visualization) with final free energy of 1.480 for a twist boundary.}}
\label{FETwistedBC}
\end{figure}

As in Table \ref{gridprogFEsimpleBC} above, a sizable majority of the Newton iteration computations are isolated to the coarsest grids, with the finest grids requiring only one Newton iteration to reach the residual tolerance limit. Therefore, most of the computational cost is also isolated to the cheaper coarse grids rather than the finer levels. Here, the total work required is approximately $1.43$ times that of assembling and solving a single linearization step on the finest grid. Without nested iteration, $22$ damped Newton steps are required on the finest grid to compute the equilibrium solution. 

\begin{table}[h!]
\centering
{\small
\begin{tabular}{|c|c|c|c|c|c|}
\hline
Grid Dim. &Newton Iter.&Init. Res.&Final Res.& Deviation in $\ltwonorm{\director}^2$ & Final Energy\\
\hline
$4 \times 4$ & 19 & 6.71e-00 & 3.97e-04 & -5.69e-05, 1.50e-04 & 1.481\\
\hline
$8 \times 8$ & 5 & 1.80e-02 & 1.84e-04 & -4.10e-06, 2.57e-06 & 1.480\\
\hline
$16 \times 16$ & 2 & 4.51e-03 & 1.80e-04 & -3.27e-07, 1.51e-07 & 1.480 \\
\hline
$32 \times 32$ & 2 & 1.13e-03 & 2.09e-14 & -1.47e-08, 6.88e-09 & 1.480\\
\hline
$64 \times 64$ & 1 & 2.82e-04 & 4.31e-11 & -9.21e-10, 4.31e-10 & 1.480\\
\hline
$128 \times 128$ & 1 & 7.05e-05 & 1.36e-12 & -5.75e-11, 2.69e-11 & 1.480\\
\hline
\end{tabular}
}
\caption{\small{Grid and solution progression for the free elastic problem with twist boundary conditions with initial and final residuals for the first-order optimality conditions, minimum and maximum director deviations from unit length at the quadrature nodes, and final functional energy on each grid.}}
\label{gridprogFETwistedBC}
\end{table}

In the final numerical run, letting $r = 0.25$ and $s = 0.95$, the boundary conditions considered are
\begin{align*}
n_1 &= 0,\\
n_2 &= \cos\big(r(\pi + 2 \tan^{-1}(X_m) -2 \tan^{-1}(X_p))\big), \\
n_3 &= \sin\big(r(\pi + 2 \tan^{-1}(X_m) -2 \tan^{-1}(X_p))\big),
\end{align*}
where $X_m=\frac{-s\sin(2\pi(x+r))}{-s\cos(2\pi(x+r))-1}$ and $X_p = \frac{-s\sin(2\pi(x+r))}{-s\cos(2\pi(x+r))+1}$. Such boundary conditions are meant to simulate nano-patterned surfaces important in current research \cite{Atherton1, Atherton2}. Even in the absence of electric fields, such patterned surfaces result in complicated director configurations throughout the interior of $\Omega$. 

A similar grid progression to the cases above is applied. The Frank elastic constants for the experiment are $K_1=1$, $K_2=.62903$, and $K_3=1.32258$. This results in $\kappa < 1$. The final solution, as well as the initial guess, are displayed in Figure \ref{FENanoBC}. Table \ref{gridprogNanoBC}, again, details the relevant output data. The computed equilibrium configuration demonstrates the expected alignment and symmetries given the patterned surfaces.

The minimized functional energy is $\mathcal{F}_1 = 3.890$, compared to the initial guess energy of $13.242$. The work required is approximately $3.06$ times that of assembling and solving a single linearization step on the finest grid. On the other hand, without nested iterations, $22$ damped Newton steps are required on the finest grid. Therefore, in all cases discussed, nested iteration is successful in significantly reducing the computational work necessary to compute an equilibrium solution. 

\begin{figure}[h!]
\centering
\begin{subfigure}{.49 \textwidth}
\centering
  \includegraphics[scale=.30]{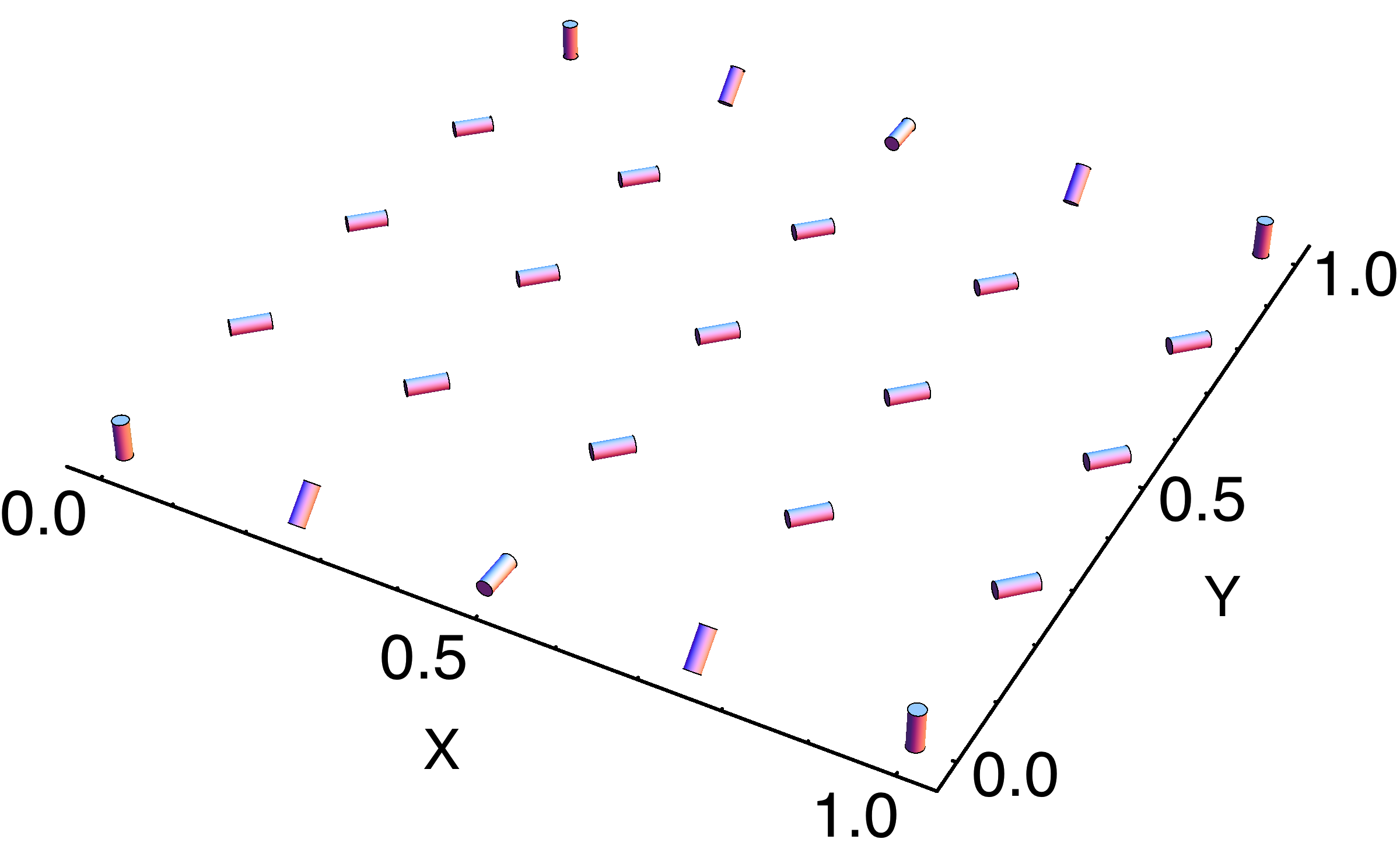}
\end{subfigure}
\begin{subfigure}{.49 \textwidth}
\raggedright
  \includegraphics[scale=.30]{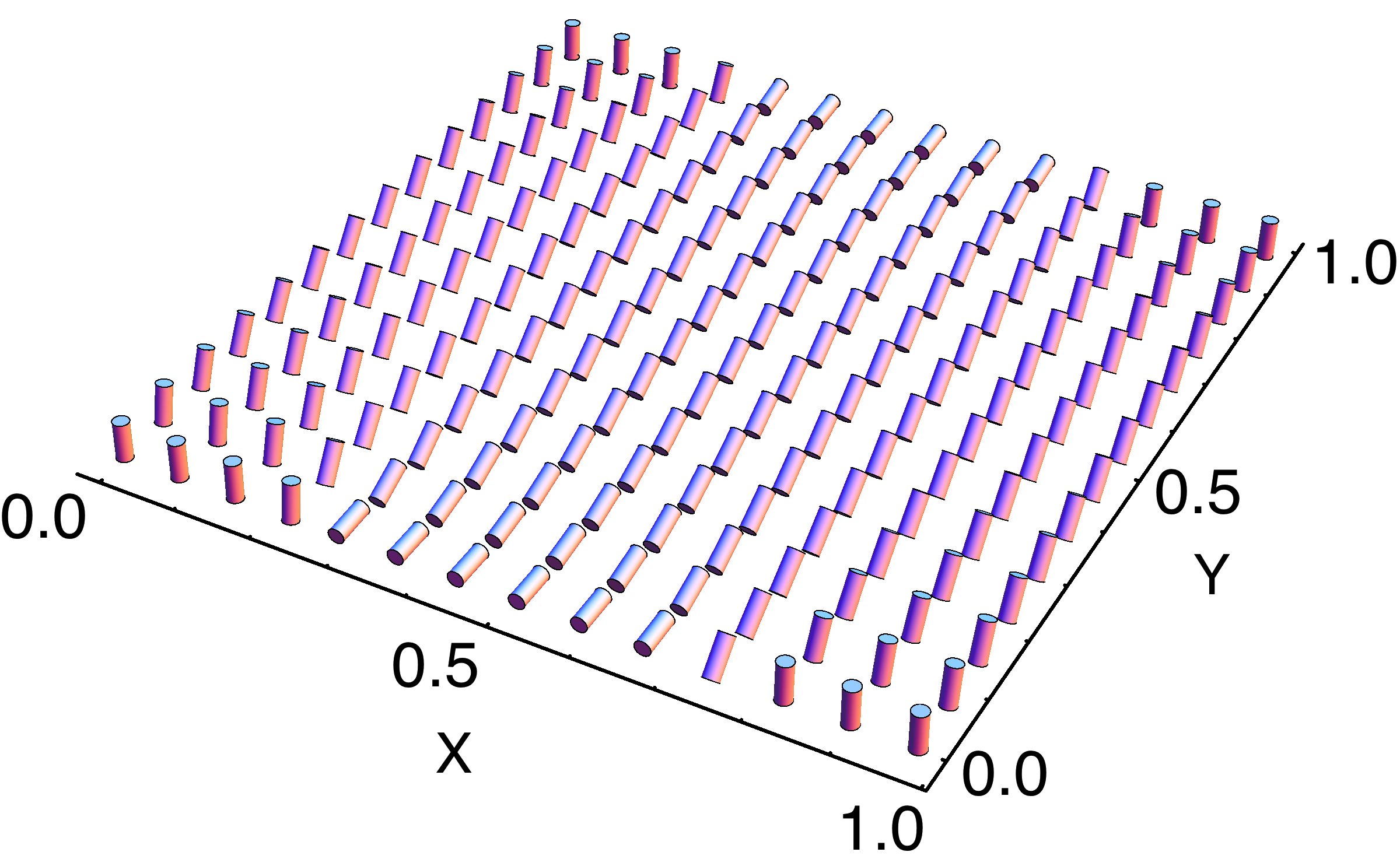}
\end{subfigure}
\caption{\small{Initial guess (left) on $4 \times 4$ mesh with initial free energy of $13.242$ and resolved solution (right) on $128 \times 128$ mesh (restricted for visualization) with final free energy of 3.890 for a nano-patterned boundary.}}
\label{FENanoBC}
\end{figure}

\begin{table}[h!]
\centering
{\small
\begin{tabular}{|c|c|c|c|c|c|}
\hline
Grid Dim. &Newton Iter.&Init. Res.&Final Res.& Deviation in $\ltwonorm{\director}^2$ & Final Energy\\
\hline
$4 \times 4$ & 19 & 7.04e-00 & 4.72e-04 & -9.07e-02, 4.67e-02 & 2.521\\
\hline
$8 \times 8$ & 9 & 1.20e-00 & 3.14e-04 & -8.20e-02, 4.58e-02 & 3.194 \\
\hline
$16 \times 16$ & 6 & 1.06e-00 & 6.71e-05 & -6.69e-02, 3.96e-02 & 3.674\\
\hline
$32 \times 32$ & 3 & 8.22e-01 &  3.42e-12 & -4.31e-02, 2.78e-02 & 3.885\\
\hline
$64 \times 64$ & 3 & 5.04e-01 & 4.75e-14 & -1.73e-02, 1.26e-02 & 3.900\\
\hline
$128 \times 128$ & 2 & 2.24e-01 & 3.00e-09 & -3.51e-03, 2.81e-03 & 3.890\\
\hline
\end{tabular}
}
\caption{\small{Grid and solution progression for patterned boundary conditions with initial and final residuals for the first-order optimality conditions, minimum and maximum director deviations from unit length at the quadrature nodes, and final functional energy on each grid.}}
\label{gridprogNanoBC}
\end{table}

\section{Summary and Future Work} \label{conclusion}

We have discussed a constrained minimization approach for liquid crystal equilibrium configurations in the presence of free elastic effects. Such minimization is founded upon the Frank-Oseen model for liquid crystal free energy. Due to the nonlinearity of the continuum first-order optimality conditions, Newton linearizations were derived. The resulting discrete systems were analyzed, and it was shown that solutions to the discretized Newton iterations exist. If $\kappa=1$ or $\kappa$ satisfies the conditions of the small data assumption in Lemma \ref{coercivitysmalldata} and the assumptions of Lemma \ref{bubblespacelemma} hold, then unique solutions to the discrete Newton iterations are guaranteed for the prescribed discrete spaces. Error analysis was conducted to demonstrate discrete convergence results for the method.

Numerical results demonstrate the accuracy and efficiency of the algorithm in resolving some difficult features for free elastic effects. The experiments address problems that include unequal Frank constants and nano-patterned boundary conditions. The experiments also reveal the necessity for a mixed finite-element approach. Such a requirement exposes an interesting parallel to other problems with similar instabilities such as the Stokes' and Navier-Stokes' equations. The minimization approach overcomes some difficulties inherent to the liquid crystal equilibrium problem, such as the nonlinear unit length director constraint, and effectively deals with heterogeneous Frank constants. The algorithm also productively utilizes nested iteration to reduce computational costs by isolating much of the computational work to the coarsest grids. Such computational work allocation significantly reduces the effective number of Newton iterations on the finest grid, even for the nano-patterned boundary conditions example.

The above method is currently being extended to include electric and flexoelectric effects in order to more accurately capture physical phenomenon important to many applications, such as the study of bistable devices \cite{Davidson1}. The rising complexity involved in these extensions presents interesting challenges, such as the appearance of more complicated saddle-point structures. Development and implementation of specifically tailored solvers for the systems encountered above, as well as those anticipated in future problems, is a priority.

Additionally, investigation into the use of $\Hn{-1}{\Omega}$ norms for the Lagrange multiplier to achieve discrete inf-sup stability independent of the mesh parameter, $h$, are being pursued. Furthermore, analysis of the Newton linearizations for the electric and flexoelectric augmentations will be undertaken. Future work will also include study of effective adaptive refinement and linearization tolerance schemes. Because the energy minimization formulation does not yield an obvious a priori error estimator, new techniques will be explored to flag cells for refinement and determine when grid refinement should occur.

\section*{Acknowledgments}
The authors would like to thank Professors Thomas Manteuffel, Johnny Guzm\'{a}n, and Ludmil Zikatanov for their useful contributions and suggestions.


\bibliographystyle{plain}	

\nocite{*}		

\bibliography{LiquidCrystalMinimization}		

\end{document}